\documentclass[9pt]{amsart}

\usepackage{lineno}

\usepackage{hyperref}
\usepackage{comment}

\usepackage{dsfont}
\usepackage{amsmath}
\usepackage{amssymb}
\usepackage{graphicx}

\usepackage{amssymb}

\usepackage[all,cmtip]{xy}

\usepackage{amsmath}

\usepackage{amsfonts}

\usepackage{soul}

\usepackage{mathpazo}
\usepackage{color}
\usepackage{yfonts}

\usepackage{paralist}

\usepackage{stmaryrd}

\usepackage{amsxtra}

\def\cU{\mathcal A}
\def\cV{\mathcal V}

\def\cB{          \mathcal B}

\def\cH{          \mathcal H}

\let\cal\mathcal

\def \R{{\mathbb R}}

\def \Z{{\mathbb Z}}

\newcommand{\T}{{\mathbb T}}
\newcommand{\prf}{{\begin{proof}}}
\newcommand{\epf}{{\end{proof}}}

\newcommand{\Q}{{\mathbb Q}}

\newcommand{\ary}{\begin{eqnarray}}
\newcommand{\eary}{\end{eqnarray}}

\newcommand{\aryst}{\begin{eqnarray*}}
\newcommand{\earyst}{\end{eqnarray*}}

\newcommand{\enmt}{\begin{enumerate}}
\newcommand{\eenmt}{\end{enumerate}}

\newtheorem{thm}{\bf Theorem}[section]

\newtheorem{lemma}[thm]{\bf Lemma}
\newtheorem{prop}[thm]{\bf Proposition}
\newtheorem{cor}[thm]{\bf Corollary}

\newtheorem{claim}[thm]{\bf Claim}

\newtheorem{quest}[thm]{\bf Question}

\newtheorem{theorem}[thm]{\bf Theorem}

\newtheorem{proposition}[thm]{\bf Proposition}

\theoremstyle{definition}
\newtheorem{definition}[thm]{\bf Definition}
\newtheorem{remark}[thm]{\bf Remark}

\DeclareMathOperator{\diff}{Diff}

\theoremstyle{definition}

\def\bee{\begin{equation}}
\def\eee{\end{equation}}

\def \cH {{\cal H}}

\def\bL{  {\bf{L}} }

\def\bPhi{  {\Phi} }

\newcommand{\Diff}{{\rm{Diff}}}

\newcommand{\pdvr}[2]
{\dfrac{\partial^{#2} #1}{\partial \theta^{#2_1} \partial r^{#2_2}}}
\newcommand{\pdvrs}[2]
{\partial^{#2} #1 /\partial \theta^{#2_1} \partial r^{#2_2}}
\newcommand{\sgn}{\mathop{\rm sgn}}

\newtheorem{example}{Example}

\numberwithin{equation}{section}

\definecolor{blue}{rgb}{0,0,1}
\definecolor{red}{rgb}{1,0,0}
\definecolor{green}{rgb}{0,.7,0}

\begin{document}

\title{\textbf{Density of mode-locking property for quasi-periodically forced Arnold circle maps}   }

	\author{Jian Wang}
	\address{Jian Wang, School of Mathematical Sciences and LPMC, Nankai University, Tianjin 300071, P. R. China}
	\email{wangjian@nankai.edu.cn}

	\author{Zhiyuan Zhang}
	\address{Zhiyuan Zhang, Institut Galil\'ee
		Universit\'e Paris 13, CNRS UMR 7539,
		99 avenue Jean-Baptiste Cl\'ement
		93430 - Villetaneuse, France}
	\email{zhiyuan.zhang@math.univ-paris13.fr}

\maketitle

\begin{abstract} 

We show that the mode-locking region of the family of quasi-periodically forced Arnold circle maps with a topologically generic forcing function is dense. This gives a rigorous verification of certain numerical observations in \cite{DGO} for such forcing functions. More generally, under some general conditions on the base map, we show the density of the mode-locking property among dynamically forced maps (defined in \cite{Zha}) equipped with a  topology that is much stronger than the $C^0$ topology, compatible with smooth fiber maps. For quasi-periodic base maps, our result generalizes the main results in \cite{ABD,  WZJ, Zha}.

\end{abstract}

\section{Introduction}\label{sec Introduction}

Quasi-periodically forced maps (qpf-maps) are natural generalizations of Schr\"odinger cocycles, which played an important role in the recent study of Schr\"odinger operator on $\Z$ with quasi-periodic potentials.
The notion of uniform hyperbolicity naturally generalizes to the so-called mode-locking property of qpf-maps. Thus, the topological genericity of mode-locking among Schr\"odinger cocycles (for a given base map) would immediately imply that for topologically generic potential, the spectrum is a Cantor subset of $\R$. We refer the readers to surveys \cite{D, Jit, Sim} on more results on Schr\"odinger operators.
It is natural to ask whether mode-locking holds generically among the set of general qpf-maps.
The first result in this direction is provided by  \cite{WZJ}. The authors showed that for a topologically generic frequency $\omega$, the set of mode-locked qpf-maps with frequency $\omega$ is residual (with respect to the uniform topology).
Their result is generalized in \cite{Zha} to any fixed irrational frequency. 
In  \cite{Zha}, the following natural generalization of the notion qpf-map is introduced (such consideration is not new, and has already appeared in \cite[Section 5]{H}).
\begin{definition}[$g$-forced maps and rotation number] \label{def dynamicallyforcedmap}
Let $\Diff^{1}(\mathbb{T})$, resp. ${\rm Homeo}(\mathbb{T})$, denote the set of orientation preserving diffeomorphisms, resp. homeomorphisms, of $\mathbb{T}$, and let  $\Diff^1(\mathbb{R})$, resp. ${\rm Homeo}(\mathbb{R})$,  denote the set of orientation preserving diffeomorphisms, resp. homeomorphisms, of $\mathbb{R}$. We denote by $\pi_{\R \to \T}$ the canonical projection from $\R$ to $\mathbb{T} \simeq \R/\Z$. We define $\Diff^{r}(\mathbb{T})$, $\Diff^{r}(\mathbb{R})$ for $r \in \Z_{>1} \cup \{ \infty \}$ in a similar way.

Given a uniquely ergodic homeomorphism $g: X \to X$, we say that $f: X \times \mathbb{T} \to  X \times \mathbb{T}$ is a {\em $g$-forced circle diffeomorphism}, resp.  {\em $g$-forced circle homeomorphism},  if there is a homeomorphism $F: X \times \R \to X \times \R$ of form
\aryst
F(x, w) = (g(x), F_x(w)), \quad \forall (x, w) \in X \times \R, 
\earyst
where $F_x \in \Diff^1(\R)$, resp. ${\rm Homeo(\R)}$, for every $x \in X$, such that $({\rm Id} \times \pi_{\R \to \mathbb{T}}) \circ  F = f \circ ({\rm Id} \times \pi_{\R \to \mathbb{T}})$. In this case, we say that $F$ is a {\em lift} of $f$.

For each lift $F$ of a $g$-forced map $f$, the following limit
\aryst
\rho(F) = \lim_{n \to +\infty} \frac{\pi_{2} F^{n}(x, w)}{n}
\earyst
exists, and is independent of $(x, w)$ (here $\pi_2$ is the canonical projection from $X \times \R$ to $\R$). Moreover, the number $\rho(f) = \rho(F) \mod 1$ is independent of the choice of the lift $F$.
\end{definition}

\begin{definition}[Mode-locking for $g$-forced maps]\label{def modelockinggforced}
Given a uniquely ergodic map $g: X \to X$, we say that a $g$-forced circle homeomorphism $f$ is {\em mode-locked} if $\rho(f') = \rho(f)$ for every $g$-forced circle homeomorphism  $f'$ that is sufficiently close to $f$ in the $C^0$-topology. We denote the set of mode-locked $g$-forced circle homeomorphism by $\cal{ML}$.
\end{definition}

The main result in \cite{Zha} says that mode-locking is a topologically generic property among the set of $g$-forced circle homeomorphism under a mild condition on $g$.
Such a result generalizes the main results in \cite{ABD2, WZJ}.

The current paper is a continuation of the above line of research. 

It is natural to ask whether mode-locking could be generic among qpf-maps with higher regularity, with respect to the smooth topology. By the main result in \cite{JW}, mode-locking can be rare in the measurable sense in many $C^1$ families. We hereby ask whether mode-locking could be generic in the topological sense.

\smallskip

\begin{quest} \label{main quest} 
Is a $C^1$ generic quasi-periodically forced circle diffeomorphism mode-locked?
\end{quest}

\smallskip

One can also ask for $C^r$-genericity for any $r\geq 2$. However, Question \ref{main quest} already seems to be far from easy to answer.
We also mention that a related question on qpf-circle maps has been asked in \cite[Question 33]{FK}, motivated by Eliasson's theorem in \cite{E}. Notice that a reducible qpf-circle map in \cite{FK} is accumulated by mode-locked qpf-maps. 

This paper is an attempt to study Question \ref{main quest}. Although we couldn’t give a direct answer to Question \ref{main quest}, we can show that mode-locking can be generic with respect to topologies which are much stronger than the $C^0$ topology. This is the content of our main theorem, Theorem \ref{thm:29}.
Thus we provide some evidence to a positive answer to Question \ref{main quest}.
It turns out that our theorem has implications for the quasi-periodically forced Arnold circle maps (see Theorem \ref{thm arnold1}), a class of maps that were previously studied numerically by physicists (see \cite{DGO}). We will elaborate on this point in Section \ref{sec arnold}.

\subsection{Statements of the main results} \label{sec Mainresults}
Let $X$ be a compact metric space. Let $g : X \rightarrow X$ be a strictly ergodic (i.e. uniquely ergodic and minimal) homeomorphism with a non-periodic factor of
finite-dimension, that is, there is a homeomorphism $\bar{g}:Y\rightarrow Y$, where $Y$ is an infinite compact subset of some Euclidean space $\mathbb{R}^d$, and there is a onto continuous map $h: X \rightarrow Y$ such that $h \cdot g = \bar{g}\cdot h$.
We will show that for $g$ satisfying some general condition, mode-locking is generic in topologies that are much stronger than the $C^0$ topology.

To state the condition we need, we introduce the following notion.
\begin{definition}\label{8.2.1}
 Let $g : X \rightarrow X$ be given as above. Let $G \subset \mathbb{R}$ be the subgroup of all $t$ such that there exist continuous maps $\phi : X \rightarrow \mathbb{R}$ and $\psi : X \rightarrow \mathbb{R}/\mathbb{Z}$ with $t = \int \phi d\nu$ and $\psi(g(x)) - \psi(x) = \phi(x) \mod  1$. We call $G(g)$ the {\it range of the Schwartzman asymptotic cycle for $g$}.
\end{definition} 

Clearly, the problem about the density of mode-locking depends on the topology we choose to put on the space of maps. In order to treat the problem with respect to different topologies in a uniform fashion, we introduce the following definition.

Given $r \in \Z_{\geq 1}$ and  $f, h \in \Diff^{r}(\T)$, we set
\aryst 
d_{C^{r}(\T, \T)}(f, h) = \sum_{k = 0}^{r} \sup_{w \in \T} d(D^{k}f(w), D^{k}h(w)).
\earyst 
We define for any $r \in \Z_{\geq 1}$ that
\aryst 
d_{\Diff^{r}(\T)}(f, h) = d_{C^{r}(\T, \T)}(f, h) + d_{C^{r}(\T, \T)}(f^{-1}, h^{-1}).
\earyst 
We define
\aryst 
d_{\Diff^{\infty}(\T)}(f, h) = \sum_{k = 1}^{\infty} 2^{-k}\frac{d_{\Diff^{k}(\T)}(f, h)}{d_{\Diff^{k}(\T)}(f, h) + 1}.
\earyst 

\begin{definition} \label{def cal H}
Given a  complete metric space ${\cal H}$ endowed with a complete metric $d_{\cal H}$, and a continuous map $\iota_{\cal H}: {\cal H} \to \Diff^1(\T)$. 
By a slight abuse of notations, we denote $\iota_\cH : \cH^n \to \Diff^1(\T)$ by $\iota_\cH((h_i)_{i=0}^{n - 1}) = \iota_\cH(h_{n - 1}) \circ \cdots \circ \iota_\cH(h_{0})$ for every integer $n \geq 1$.
\end{definition}

In the rest of the paper, we will only consider two classes of tuples $(\cH, d_\cH, \iota_\cH)$ given below.

\begin{example} \label{exam I} 
Let $r \in \Z_{\geq 1} \cup \{ \infty \}$.
Let $\cH = \widetilde{\Diff^{r}(\T)} = \{ f \in \Diff^{r}(\R) \mid f(y + 1) = f(y) +1, \forall y \in \R \}$.
We may define distance $d_{\widetilde{\Diff^{r}(\T)}}$ on $\widetilde{\Diff^{r}(\T)}$ analogously to $d_{\Diff^{r}(\T)}$.
Let 
\aryst
d_{\cH} = d_{\widetilde{\Diff^{r}(\T)}}.
\earyst
We define $\iota_\cH: \cH \to \Diff^{1}(\T)$ by
\aryst
\iota_\cH(h)(y \mod 1) = h(y) \mod 1, \ \forall y \in \R.
\earyst
\end{example}

\begin{example}  \label{exam III}
Let  $\cH = \R$.
Let $P \in C^{\omega}(\T)$ be a non-constant real analytic function such that $\| P' \| < 1$.
We define $\iota_\cH: \cH \to \Diff^{1}(\T)$ by
\aryst
\iota_\cH(h)(w) = w + P(w) + h \mod 1, \ \forall w \in \R/\Z.
\earyst
We let $d_{\cH} = d_{\R}$ where $d_\R$ denotes the Euclidean distance.
\end{example}

\begin{remark} \label{rem comparenorms}

By definition, for $(\cH, d_\cH, \iota_\cH)$ in either Example \ref{exam I} or \ref{exam III},  we know that there exists some constant $C > 0$ such that for any $h_1, h_2 \in {\cal H}$, we have $$C d_{\cal H}(h_1, h_2) \geq d_{\Diff^{1}(\mathbb{T})}(\iota_{\cH}(h_1), \iota_{\cH}(h_2)).$$
\end{remark}

 As usual, let us denote by $C^0(X, {\cal H})$ the collection of continuous maps from $X$ to ${\cal H}$. We equipe $C^0(X, \cH)$ with the norm
\aryst
D_{\cH}(H,  H') = \sup_{x \in X}d_{\cH}(H(x), H'(x)).
\earyst
Given any $H \in C^{0}(X, {\cal H})$ and any $\epsilon > 0$, we denote
\aryst
\cB_{\cal H}(H, \epsilon) = \{ H' \in C^{0}(X, {\cal H}) \mid D_{\cal H}(H, H') < \epsilon  \}.
\earyst

We denote by $\Diff^{0,1}_g(X \times \mathbb{T})$ the collection of $g$-forced circle diffeomorphism of form
\aryst
f: X \times \mathbb{T} &\to& X \times \mathbb{T}, \\
(x, w) &\mapsto& (g(x), f_x(w))
\earyst
where $f_x \in \Diff^{1}(\T)$ depends continuously on $x \in X$. We denote
\aryst
\|f\|_{C^{0,1}} &=& \sup_{x \in X} (\|Df_x\|, \|D(f_x^{-1})\|) < \infty, \\
D_{C^{0,1}}(f, f') &=&   \sup_{x \in X} d_{\Diff^{1}(\T)}(f_x, (f')_x) < \infty.
\earyst
By Remark \ref{rem comparenorms}, there is a continuous map $\bPhi : C^0(X, {\cal H}) \to \Diff^{0,1}_g(X \times \mathbb{T})$ defined by
\aryst
\bPhi(H)(x, w) = (g(x),  \iota_{\cH}(H(x))(w)).
\earyst

The main theorem of this paper is the following.
\begin{theorem}\label{thm:29}
Let $({\cal H}, d_{\cal H}, \iota_{\cal H})$ be given by Example \ref{exam I}  or \ref{exam III}.  If $ G(g)$ is dense in $\mathbb{R}$, then 
\ary \label{def MLH}
\cal{ML}(\cH, \iota_{\cH}) &:=& \{ H \in C^0(X, {\cal H}) \mid \bPhi(H) \in \cal{ML} \}
\eary
is open and dense in $C^0(X, {\cal H})$.
\end{theorem}
When $\iota_{\cH}$ is clear from the context, we will abbreviate $\cal{ML}(\cH, \iota_{\cH})$ as $\cal{ML}(\cH)$.

 Instead of Example \ref{exam III}, it is more convenient to consider the following.
\begin{example}  \label{exam II}
Let  $P$ and $(\cH, d_{\cH}, \iota_{\cH})$ be as in Example \ref{exam III}. We assume in addition that $P$ has no smaller period, i.e., there exists no constant $\rho \in (0,1)$ such that $P(w + \rho) \equiv P(w)$.
\end{example}
We will actually give the proof for the following theorem, Theorem \ref{thm:29'}, since the discussions for Example  \ref{exam I}  and  \ref{exam II} can be organized in a unified way.   In Section \ref{sec Density of mode-locking}, we will prove Theorem \ref{thm:29'}, and deduce Theorem \ref{thm:29} as a corollary.
\begin{theorem}\label{thm:29'}
Let $({\cal H}, d_{\cal H}, \iota_{\cal H})$ be given by Example  \ref{exam I}  or  \ref{exam II}.  If $ G(g)$ is dense in $\mathbb{R}$, then  $\cal{ML}(\cH) = \cal{ML}(\cH, \iota_{\cH}) $ is open and dense in $C^0(X, {\cal H})$.
\end{theorem}

We note that $G(g)$ is dense for many interesting maps $g$, such as
\begin{enumerate}
    \item  minimal translations of $\mathbb{T}^d$ for any $d \geq 1$,
    \item the skew-shift $(x, y)  \mapsto (x + \alpha, y + x)$ on $\mathbb{T}^2$, where $\alpha$ is irrational,
    \item any strictly ergodic homeomorphism on a totally disconnected infinite compact subset of $\mathbb{R}^d$  for any $d\geq 1$.
\end{enumerate} 

The proof of Theorem \ref{thm:29} is based on a description about the complement of $\overline{\cal{ML}(\cH)}$  (see Theorem \ref{Theorem 30}. Before stating the result we introduce the following notions.

For a $g$-forced circle homeomorphism $f$ with a lift $F$, for any integer $n \geq 1$, we denote as in \cite{Zha} that
\aryst
(f^n)_x :=  f_{g^{n-1}(x)} \circ \cdots \circ f_x, \quad (F^n)_x := F_{g^{n-1}(x)} \circ \cdots \circ F_x.
\earyst

\begin{definition}\label{DEFINITION 8.2.3.}
 For any $f \in \Diff^{0,1}_g(X \times\mathbb{T})$, the {\it extremal fiberwise Lyapunov exponents}  of $f$, denoted by $L_{+}(f)$ and $L_{-}(f)$, are given by formulas
\[ \begin{array}{ccc}
   L_{+}(f) &:= & \lim\limits_{n\rightarrow\infty} \frac{1}{n} \sup\limits_{(x ,w)\in X\times\mathbb{T}} \log D(f^n)_x(w),\\
   L_{-}(f) & :=& \lim\limits_{n\rightarrow\infty} \frac{1}{n} \inf\limits_{(x ,w)\in X\times\mathbb{T}} \log D(f^n)_x(w).
 \end{array}
\]
Let $({\cal H}, d_{\cal H}, \iota_{\cal H})$ be in Example \ref{exam I}  or \ref{exam II}. Then for any $H \in C^0(X, {\cal H})$, we denote
\aryst
L_{+}(H) = L_{+}(\Phi(H)), \quad L_{-}(H) = L_{-}(\Phi(H)).
\earyst
\end{definition}

\begin{theorem}\label{Theorem 30}
Let $({\cal H}, d_{\cal H}, \iota_{\cal H})$ be given by Example \ref{exam I}  or \ref{exam II}.
 Then either $\cal{ML}(\cH)$ is dense in $C^0(X, {\cal H})$, or there exists a residual subset $\cU$ of $C^0(X, {\cal H}) \setminus \overline{\cal{ML}(\cH)}$
 such that $\bPhi(H)$ has zero extremal fiberwise Lyapunov exponents for any $H \in \cU$. Here
the notion of a residual subset is defined with respect to the distance $D_{\cal H}$ on $C^0(X, {\cal H})$.
\end{theorem} 
 Theorem \ref{Theorem 30} is similar to \cite[Theorem 5]{ABD2}, where the authors have shown that 
 any cocycle in $C^{0}(X, { SL}(2, \R))$ that is not uniformly hyperbolic can be approximated by cocycles that are conjugate to elements in $C^{0}(X, {SO}(2, \R))$. A similar approach is taken in \cite{Zha} to prove the density with respect to the $C^0$ topology.

 \subsection{Application to Quasi-periodically forced Arnold circle maps} \label{sec arnold}

A prominent example of qpf-map is the so called \lq\lq quasi-periodically forced Arnold circle map \rq\rq,
\aryst
f_{\alpha, \beta, \tau, q}: \T^2 \to \T^2, \quad (\theta,x) \mapsto (\theta+\omega, x + \tau + \frac{\alpha}{2\pi}\sin(2\pi x) + \beta q(\theta) \mod \ 1),
\earyst
with parameters $\alpha \in [0,1], \tau, \beta \in \R$ and a continuous forcing function $q : \T \to \R$. It can be traced back to \cite{GOPY}, and was then studied in \cite{RO, BOA, DGO}, etc., as a simple model of an oscillator forced at two incommensurate frequencies. It has served as a source of motivation for a series of articles on this subject, such as \cite{BJ, J, J2, JW, KWYZ, Zha}, etc.

Mode-locking was observed numerically on open regions in the $(\tau, \alpha)$-parameter space, known as the Arnold tongues. An immediate question is whether for any given function $\beta q$, mode-locking property holds for an open and dense set of $(\tau, \alpha)$-parameters. There are numerical evidences to support such conjecture. 
In \cite[V. Conclusions]{DGO}, the authors wrote:{  \it \lq\lq  Various numerical experiments are performed to illustrate the different types of attractors that can arise in typical quasiperiodically forced systems. The central result is that in the two-dimensional parameter plane of $K$ and $V$, the set of parameters, at which the system Eqs. (2) and (3) exhibits strange nonchaotic attractors, has a Cantor-like structure, and is embedded between two critical curves.\rq\rq} Notice that $(K, V)$ in \cite{DGO} corresponds to $(\tau, \alpha)$ in our paper; and a parameter at which a strange nonchaotic attractor in  \cite{DGO}  appears is in the complement of the mode-locking region. The following theorem gives a rigorous verification of the observation that the set of $(\tau, \alpha)$ where  strange nonchaotic attractors appear contains no interior for a topologically generic forcing function.
\begin{theorem}\label{thm arnold1}
For any $\omega \in (\R \setminus \Q)/\Z$ and any non-zero $\beta \in \R$, there is a residual subset ${\cal U}_{\beta} \subset C^0(\T)$ such that for every $q \in {\cal U}_{\beta}$, the set $\{ ( \tau, \alpha) \mid f_{\alpha, \beta, \tau, q} \in \cal{ML}  \}$
is an open and dense subset of $\R \times (0,1)$.
\end{theorem}
\begin{remark}
We can clearly see from the proof below that a similar result holds when the function $\sin(2\pi x)$ is replaced by {\em any} non-constant real analytic function on $\T$.
\end{remark}
\begin{proof}
Fix some $\omega \in (\R \setminus \Q)/\Z$ and define $g: \T \to \T$ by $g(x) =  x + \omega$.
Fix some $\tau \in \R$, $\alpha \in (0,1)$ and $\beta \in \R \setminus \{0\}$. 
We let ${\cal H}_{\alpha, \beta, \tau} = \R$ and let $d_{\cH_{\alpha, \beta, \tau}}$ be the Euclidean metric. We take
\aryst
\iota_{\cH_{\alpha, \beta, \tau}}( h ) = (x \mapsto x + \tau + \frac{\alpha}{2\pi}\sin(2\pi x) + \beta h \mod \ 1).
\earyst
It is clear that $( \cH_{\alpha, \beta, \tau}, d_{\cH_{\alpha, \beta, \tau}}, \iota_{\cH_{\alpha, \beta, \tau}})$ belongs to Example \ref{exam III}.
Thus we can apply Theorem \ref{thm:29}  to deduce that: for every $\beta \in \R \setminus \{0\}$, for every $\tau \in \R$, $\alpha \in (0,1)$,
the set of $q \in C^0(\T)$ such that $f_{\alpha, \beta, \tau, q}$ is mode-locked is open and dense with respect to the uniform topology on $C^0(\T)$. Then we take a dense subset $\{ (\tau_n, \alpha_n) \}_{n \geq 0}$ in $\R \times (0,1)$, and for each $n \geq 0$ we let $\cB_n$ denote the set of $q \in C^0(\T)$ such that $f_{\alpha_n, \beta, \tau_n, q}$ is mode-locked. Then $\cB_n$ is open and dense for each $n \geq 0$. Consequently, the set ${\cal U}_\beta = \cap_{n \geq 0} \cB_n$ is a residual subset of $C^0(\T)$, and for every $q \in {\cal U}_\beta$, the open set $\{ ( \tau, \alpha) \mid f_{\alpha, \beta, \tau,  q} \in \cal{ML}  \}$ is dense since it contains $(\tau_n, \alpha_n)$ for every $n \geq 0$.  This completes the proof.
\end{proof}

\subsection{Idea of the proof} 
As we have explained, the key step is Theorem \ref{Theorem 30}. Although the general idea for proving Theorem \ref{Theorem 30} is originated from \cite{AB, B}, relying on the fact that we can, by arbitrarily small perturbation, promote linear displacement for the orbit of a given point, the execution in our case is more complicated. Compared to the case of $SL(2,\mathbb{R})$-cocycles, where the matching of (temporary) stable and unstable directions automatically reduces the growth of the derivative for every other point on the circle, our general $g$-forced circle diffeomorphisms have no such convenient feature. We overcame this problem by carefully dividing a circle fiber into two parts so that for one part we prove certain cancellation in the future, and for the other in the past. For each part, the desired cancellation is also obtained for different reasons, under two distinct possibilities. After proving Theorem \ref{Theorem 30}, we will use it as a starting point, and perform a further perturbation to obtain the genericity of mode-locking. For this part, we need to construct a global perturbation by combining several local perturbations at different scales.  
This is organized through the stratification introduced in \cite{ABD2}.
For the above strategy to work, we need to use some features of maps in Example \ref{exam I} and \ref{exam II} to ensure that ${\cal H}$ is sufficiently rich so that we can produce certain parabolic-like and hyperbolic-like circle diffeomorphisms by making perturbations in ${\cal H}$ (see Lemma \ref{lem arnoldfamilyismovable} and Lemma \ref{lem getparabolicelement}).

\subsection{Notations}

Given an integer $k \geq 1$ and $g_1,\cdots,g_k \in \diff^{1}(\T)$ (or $\diff^1(\R)$), we denote by 
$\Pi_{i=1}^k g_i$ the map $g_k \circ \cdots \circ g_1$.

Throughout this paper, in all the lemmas and propositions, we will assume that all the constants depend on some $({\cal H}, d_{\cal H}, \iota_{\cH})$ fixed throughout this paper. For the sake of simplicity, we will not explicitly present such dependence.

\section{Preliminary}

We fix some $({\cal H}, d_{\cal H}, \iota_{\cH})$  in Example \ref{exam I}  or \ref{exam II} throughout this section.  We let $g: X \to X$ be a map given at the beginning of Section \ref{sec Mainresults}.

\subsection{Basic properties of dynamically forced maps}

Let us recall some basic properties of $g$-forced circle homeomorphisms with a uniquely ergodic $g$, proved in \cite{ABD2, Zha}. In particular, all the results in this subsection apply to maps in $ \Diff^{0,1}_g(X\times\mathbb{T})$.

Let $f$ be a $g$-forced circle homeomorphism with a lift $F$. Then for any $g$-forced circle homeomorphism $f'$ such that
$D_{C^{0,1}}(f',f)<\frac{1}{2}$, there exists a unique lift of $f'$, denoted by $F'$, such that $d_{\Diff^1(\R)}(F',F) <  \frac{1}{2}$. In this rest of this paper, we will say such $F'$ is the lift of $f$ that is {\it close} to $F$.

We first notice the following alternative characterization of the mode-locking property (see \cite[Definition 3 and Lemma 2]{Zha}).
\begin{lemma} \label{claim equivdefmodelock}
 A $g$-forced circle homeomorphism $f$ is mode-locked if and only if there exists $\epsilon > 0$ such that 
 we have 
 \aryst
 \rho(F_{-\epsilon}) = \rho(F) = \rho(F_{\epsilon}),
 \earyst
 where $F$ is an arbitrary lift of $f$, and $F_t(x, y) = F(x, y) + t$ for all $(x, y) \in X \times \R$ and $t \in \R$.
 \end{lemma}

\begin{definition}
For any $f \in  \Diff^{0,1}_g(X \times \mathbb{T})$, for any lift of $f$ denoted by $F$, for any integer $n > 0$ we set
\aryst
\underline{M}(F, n) = \inf_{(x, y) \in X \times \R} ((F^N)_x(y) - y), \quad \overline{M}(F, n) = \sup_{(x, y) \in X \times \R}((F^N)_x(y) - y).
\earyst
\end{definition}

The following is \cite[Lemma 3]{Zha}.
\begin{lemma}\label{LEMMA 55}
Given some $f \in \Diff^{0,1}_g(X \times \mathbb{T})$ and a lift of $f$ denoted by $F$. For any $\kappa_0 > 0$, there exists $N_0 = N_0(f, \kappa_0) > 0$ such that for any $n > N_0$, we have
$[\underline{M}(F, n), \overline{M}(F, n)] \subset n \rho(F) + (-n\kappa_0, n\kappa_0)$.
\end{lemma}

The following is an immediate consequence of Lemma \ref{LEMMA 55}.
\begin{cor}\label{LEMMA 56}
Given some $f \in \Diff^{0,1}_g(X \times \mathbb{T})$ and a lift of $f$ denoted by $F$. If for some $\epsilon > 0$ we have $\rho(F_{-\epsilon}) < \rho(F)$, resp. $\rho(F) < \rho(F_{\epsilon})$, then there exist $\kappa_1 = \kappa_1(f, \epsilon) > 0$ and  an integer $N_1 = N_1(f, \epsilon) > 0$ such that for any $n > N_1$, we have
\aryst
\overline{M}(F_{-\epsilon}, n) < \underline{M}(F, n) - n \kappa_1, \mbox{ resp. } \overline{M}(F, n) + n \kappa_1  < \underline{M}(F_\epsilon, n).
\earyst
\end{cor}

Given a $g$-forced circle homeomorphism $f$, and a lift of $f$ denoted by $F$. For any integer $N > 0$, any $\kappa > 0$, we define 
\ary \label{eq omegaNFkappa}
 \Omega_N(F, \kappa)=\{(x,y,z)\in X\times\mathbb{R}^2\mid  |z- (F^N_x)(y)|< N \kappa \}.
 \eary

We recall the following result, which is analogous to \cite[Lemma 8]{Zha}.  
\begin{lemma} \label{lem cancellation} 
Given $H \in C^0(X, \cal{H})$. We let $f = \Phi(H)$ and let $F$ be a lift of $f$.  For any $\epsilon \in (0, 1/4)$ there exists $\epsilon_0 = \epsilon_0(H, \epsilon) \in (0, \epsilon)$ such that if we have $\rho(F_{\epsilon_0}) > \rho(F) > \rho(F_{- \epsilon_0})$, then by letting $\kappa_1 = \kappa_1(f, \epsilon_0)> 0$ and $N_1 = N_1(f, \epsilon_0) > 0$ be as in Corollary \ref{LEMMA 56}, the following is true. For any $\check{H} \in \cB_{\cal H}(H, \epsilon_0)$, for any integer $N \geq N_1$, there exists a continuous map $\Phi^{\check{H}}_N : \Omega_N(F, \kappa_1) \to \cal{H}^{N}$ such that for any $(x, y, z) \in \Omega_N(F, \kappa_1)$,  let $\Phi^{\check{H}}_N(x, y, z) = (p_0,\cdots, p_{N-1})$ and let $\check{F}$ be the unique lift of $\Phi(\check{H})$ close to $F$\footnote{The map $\check{F}$ is well-defined since $\epsilon_0 < \epsilon < 1/2$.}, then
\enmt
\item $d_{\cal{H}}(p_i, \check{H}(g^{i}(x))) < 2\epsilon$ for every $0 \leq i \leq N-1$, 
\item let $P_i$ be the unique lift of $\iota_{\cal H}(p_i)$ close to $\check{F}_{g^{i}(x)}$, then $P_{N-1} \cdots P_0(y) = z$,
\item if $z = (\check{F}^{N})_{x}(y)$, then $p_i = \check{H}(g^{i}(x))$ for every $0 \leq i \leq N-1$.
\eenmt
\end{lemma}
\begin{proof} 
Firstly, we see that $\Omega_N(F, \kappa_1) $ is the same as $\Gamma_N(F, \kappa_1)$ defined in \cite[Section 4.1]{Zha}.
Secondly, we notice that the following is true for any $(\cH, d_\cH, \iota_\cH)$ in Example \ref{exam I} or \ref{exam II}: 
for any $p \in {\cal H}$, for any lift $P$ of $\iota_{\cal H}(p)$, we have a map $\varphi : (-1,1) \to {\cal H}$ such that $\varphi(0) = p$ and for every $t \in (-1,1)$, $P_t$ a lift of $\iota_{\cal H}(\varphi(t))$. Moreover, we may let $\varphi$ depend continuously on $p$.
Then we can follow the same argument in the proof of \cite[Lemma 8]{Zha} to deduce the lemma. 
\end{proof}
We have the following.
\begin{prop}[Tietze's Extension Theorem for $\cH$]\label{prop-tietze}
Let $H \in C^0(X, {\cal H})$ and let $Y$ be a compact subset of $X$.
Given a constant $\epsilon > 0$ and a continuous map $H_0: Y \to {\cal H}$ such that for every $x \in Y$ we have $d_{\cal H}(H(x), H_0(x)) < \epsilon$. Then there exists $H_1 \in  \cB_{\cal H}(H, \epsilon)$ such that $H_1(x) = H_0(x)$ for every $x \in Y$.
\end{prop}
\begin{proof}
 If $(\cH, d_\cH, \iota_\cH)$ is in Example \ref{exam II}, then this is the classical Tietze's Extension Theorem.
Now we assume that $(\cH, d_\cH, \iota_\cH)$ is in Example \ref{exam I}.  Notice that we can naturally identify $\widetilde{\Diff^r(\T)}$ with
$\cV_{r} = \{ \varphi \in C^r(\T) \mid  \varphi'(w) > - 1, \forall w \in \T \}$ since we can associate to each $\varphi \in \cV_r$ a mapping $y \mapsto y + \varphi(y \mod 1)$ in $\widetilde{\Diff^r(\T)}$, and vice versa. Since $\cV_r$ is a convex open subset of $C^r(\T)$ which is a locally convex topological vector space, we can apply the generalization of Tietze's Extension Theorem for such spaces in \cite{Dug} to conclude the proof.
 \end{proof}

\subsection{Lyapunov exponents}
 
 The following is an elementary yet useful observation.

\begin{lemma}\label{LEMMA 65}
For any $f \in\Diff^{0,1}_g(X\times\mathbb{T})$, we have 
$L_{+}(f)\geq 0\geq L_-(f)$ and $L_+(f),-L_-(f)\leq \log   \Vert f \Vert_{C^{0,1}}$.
\end{lemma}
\begin{proof} 
Since for any $n \geq 1$, any $x\in  X$, we have
 \begin{equation}
\int_\mathbb{T} D(f^n)_x(w)dw = 1.  \label{(8.3.1)}
\end{equation}
 We necessarily have that $L_-(f) \leq 0 \leq L_+(f)$. By Definition \ref{DEFINITION 8.2.3.} and the chain rule, we clearly have
$\max(L_+(f),-L_-(f)) \leq \log\Vert f\Vert_{C^{0,1}}$.
\end{proof}

%
%

Given an integer $N > 0$, we define for each ${\bf h} = (h_0, \cdots, h_{N-1}) \in {\cal H}^N$ that
\aryst
{\bf L}_+({\bf h}) &=&  \frac{1}{N} \sup\limits_{w \in \T} \log D(\iota_{\cH}(h_{N-1}) \cdots \iota_\cH(h_0))(w), \\
{\bf L}_-({\bf h})  &=&  \frac{1}{N} \inf \limits_{w \in \T} \log D(\iota_{\cH}(h_{N-1}) \cdots \iota_\cH(h_0))(w).
\earyst
By definition, we have
\aryst
L_{\pm}(H) = \lim_{n \to \infty} {\bf L}_{\pm}( (H(g^{i}(x)))_{i=0}^{n-1} ).
\earyst
 
 The following lemma, whose proof we omit, follows immediately from the subadditivity, e.g., $\log \| D(f_1 f_2) \| \leq \log \| D f_1 \| + \log \| D f_2 \|$.
 \begin{lemma} \label{LEMMA limitLE}
 Given  $H \in C^0(X, {\cal H})$. For any $\kappa_0 > 0$, there exists $N'_0 = N'_0(H, \kappa_0) > 0$ such that for any $n > N'_0$ and any $x \in X$, we have
\aryst
[ {\bf L}_-( (H(g^{i}(x)))_{i=0}^{n-1} ),  {\bf L}_+( (H(g^{i}(x)))_{i=0}^{n-1} ) ] 
\subset  ( L_-( H ) - \kappa_0, L_+( H ) + \kappa_0).
\earyst
 \end{lemma}

\section{A criterion from stratification} \label{secKZ}

Let $g : X \rightarrow X$ be given at the beginning of Section \ref{sec Mainresults}. That is, $X$ is a compact metric space, and $g$ is strictly ergodic with a non-periodic factor of finite dimension.

\subsection{Dynamical stratification} \label{subsec dynstrat}

As in \cite{ABD2}, for any integers $n,N,d > 0$,  a compact subset $K \subset X$ is said
\enmt
\item \textit{$n$-good} if $g^{k}(K)$ for $0 \leq k \leq n-1$ are disjoint subsets,
\item \textit{$N$-spanning} if the union of $g^{k}(K)$ for $0 \leq k \leq N-1$ covers $X$,
\item \textit{$d$-mild} if for any $x \in X$, $\{g^{k}(x) \mid k \in \Z\}$ enters $\partial K$ at most $d$ times.  
\eenmt

The following is an immediate consequence of \cite[Lemma 5.2 - Lemma 5.4]{ABD2}.
\begin{lemma} \label{LEMMA 58}
There exists an integer $d > 0$ such that for every integer $n > 0$, for every open set $U \subset X$, there exist an integer $D > 0$ and a compact subset $K \subset U$  that is  $n$-good, $D$-spanning and $d$-mild.
\end{lemma}

Let  $K \subset X$ be a  $n$-good, $M$-spanning and $d$-mild compact subset. For each $x \in X$, we set
\aryst
&&l^{+}(x) = \min\{j > 0 \mid g^{j}(x) \in int(K)\}, \quad l^{-}(x) = \min\{j \geq 0 \mid g^{-j}(x) \in int(K) \}, \\
&&l(x) = \min\{j > 0 \mid g^{j}(x) \in K \}, \\
&&T(x) = \{j \in \Z \mid - l^{-}(x)  < j < l^{+}(x)\}, \quad T_B(x) = \{j \in T(x) \mid g^{j}(x) \in \partial K \}, \\
&&N(x) = \#T_B(x), \quad K^{i} = \{x \in K \mid N(x) \geq d-i\},\forall -1 \leq i \leq d.
\earyst
Let $Z^{i} = K^{i} \setminus K^{i-1} = \{x \in K \mid N(x) = d-i\}$ for each $0 \leq i \leq d$.

Lemma \ref{LEMMA 58} and the notations introduced above are minor modifications of those in the proof of \cite[Lemma 4.1]{ABD2}. We also have the following (see \cite{ABD2} and also  \cite[Lemma  7]{Zha}).
\begin{lemma}\label{LEMMA 59} We have
\enmt 
\item For any $x \in K$, $l(x) \leq l^{+}(x)$ and $n \leq l(x) \leq D$,
\item $T$ and  $T_B$ are upper-semicontinuous,
\item $T$ and $T_B$, and hence also $l$, are locally constant on $Z^{i}$,
\item $K^i$ is closed for all $-1 \leq i \leq d$ and $\emptyset = K^{-1}  \subset K^{0} \subset \cdots \subset K^{d} = K$,
\item For any $x \in K^{i}$, any $0 \leq m < l^{+}(x)$ such that $g^{m}(x) \in K$, we have $g^{m}(x) \in K^{i}$.
\eenmt
\end{lemma}

\subsection{A criterion for mode-locking}

Given a $g$-forced circle homeomorphism $f$, a lift of $f$ denoted by $F$, and a compact subset $K \subset X$ that is $M$-spanning for some $M > 0$, let
\aryst
  f_K(x,w)& = &f^{l(x)}(x,w), \ \forall(x,w)  \in K \times \mathbb{T},\\
  F_K(x,y)& =& F^{l(x)}(x,y), \ \forall(x,y) \in K \times \mathbb{R}.   
\earyst
We have the following sufficient condition for mode-locking.

\begin{lemma}\label{lemma 66} If there exists an open set ${\cal R} \subset   K\times\mathbb{T}$ (with respect to the induced topology on $K \times \mathbb{T}$)
 such that for each $x\in K$, we have ${\cal R} \cap ( \{ x \}\times \mathbb{T})= \{ x \} \times I_x$ for some non-empty open interval $ I_x \subsetneqq\mathbb{T}$, and $f_K(\overline{\cal{R}})\subset \cal{R}$, then $f \in \mathcal{ML}$.
\end{lemma}
\begin{proof}
 Since $K$ is $M$-spanning and $f_K(\overline{{\cal R}})\subset {\cal R}$, there exists $\delta = \delta(f,M,\mathcal{R})> 0$ such that for any $\epsilon\in (-\delta, \delta)$, we have $(f_\epsilon)_K(\overline{{\cal R}}) \subset \mathcal{R}$. We inductively define a sequence of functions as follows:
  \aryst
     l_n(x) =
     \begin{cases}
       l(x),   & n=0, \\
       l(g^{\Sigma^{n-1}_{j=0} l_j(x)}(x)), & n > 0. 
       \end{cases}
\earyst
 Then it is direct to show, by an induction on $n$, that for any $x \in K$, $y\in\mathbb{R}$ so that $(x,y \text{ mod } 1)\in {\cal R}$, we have
$$ 1> |\pi_2((F_\epsilon)_K)^n(x,y) - \pi_2(F_K)^n(x,y)|=|(F^{\Sigma^{n-1}_{i=0} l_i(x)}_\epsilon)_x(y) ) - (F^{\Sigma^{n-1}_{i=0} l_i(x)})_x(y)|, \ \forall n\geq 0,$$
where $\pi_2: X \times\mathbb{R} \rightarrow \mathbb{R}$ is the canonical projection. This implies the $\rho(F_\epsilon) = \rho(F)$
for all $\epsilon \in (- \delta, \delta)$ and thus concludes the proof.
\end{proof}

\section{Density of zero Lyapunov exponent}

The goal of this section is to prove Theorem \ref{Theorem 30}. 
Throughout this section, we assume that  $(\cH, d_\cH, \iota_\cH)$ is in Example \ref{exam I} or \ref{exam II}.

We first introduce the following notion.
\begin{definition} \label{def contractable}
We say that $H \in C^{0}(X, {\cal H})$ is {\it contractable} if 
\enmt
\item either $(\cH, d_\cH, \iota_\cH)$ is  in Example \ref{exam I}, 
\item or $(\cH, d_\cH, \iota_\cH)$ is  in Example \ref{exam II}, and  there exists an integer $k > 1$ such that for any $x \in X$ and any $w \in \T$, there exists $0 \leq i \leq k - 2$ such that $P''((f^{i + 1})_{x}(w)) \neq 0$, where $f = \Phi(H)$.
\eenmt
\end{definition}

\begin{lemma} \label{lem densityofcontractable}
The set of contractable $H \in C^{0}(X, {\cal H})$ is open and dense in $C^{0}(X, {\cal H})$ with respect to the metric $D_{\cal H}$.
\end{lemma}
\begin{proof}
The openness is clear from the definition. To show the density, we will show that given any $H \in C^{0}(X, {\cal H})$ and an arbitrary $\epsilon > 0$, we can construct some $H' \in  C^{0}(X, {\cal H})$ such that $H'$ is contractable and $D_{\cal H}(H, H') < \epsilon$.

If $(\cH, d_\cH, \iota_\cH)$ is  in Example \ref{exam I}, then it suffices to take $H' = H$.

Now assume that $(\cH, d_\cH, \iota_\cH)$ is  in Example \ref{exam II}.
Denote $A = \{ w \in \T \mid P''(w) \neq 0 \}$. Since $P$ is a non-constant analytic function, $A$ is a finite set.
For any $\epsilon > 0$, there exists $\check{H} \in \cB_{\cal H}(H, \epsilon)$ such that, by denoting $\check{f} = \Phi(\check{H})$, we have $\check{f}_{x_0}(A) \cap A = \emptyset$ for some $x_0 \in X$. By continuity, for every $x$ sufficiently close to $x_0$, we have $\check{f}_{x}(A) \cap A = \emptyset$. We can then easily deduce that $\check{H}$ is contractable by the minimality of $g$.
\end{proof}

 The following lemma gives the key property of  a contractable element that we will use.  We will only need this lemma in Section \ref{sec Negative Lyapunov exponent}, and the readers can skip it during the first reading, and come back here later.
\begin{lemma} \label{lem arnoldfamilyismovable}
If  $H \in C^{0}(X, {\cal H})$ is contractable, then the following holds:
 there exist some $\epsilon > 0$, an integer $k > 0$ and a continuous map $E : [0, 1] \times \T \times X \times \cB_{\cH}(H, \epsilon) \to \cH^k$
 such that if we denote $E(\sigma, w, x, \check{H} ) = ( h^{\sigma, w}_{i} )_{i=0}^{k-1}$ then we have:
 \enmt
 \item $h^{0, w}_i = \check{H}(g^i(x))$ for every $w \in \T$ and every $0 \leq i \leq k-1$,
 \item $\iota_\cH(( h^{\sigma, w}_{i} )_{i=0}^{k-1})(w) = \iota_\cH(( \check{H}(g^i(x)) )_{i=0}^{k-1})(w)$,
 \item for every $\sigma_0 \in (0, 1)$ there exist $r_0, \epsilon_2 > 0$ such that 
 \aryst
 &&D(\iota_\cH(( h^{\sigma_0, w}_{i} )_{i=0}^{k-1}))(y) < e^{-\epsilon_2} D(\iota_\cH(( \check{H}(g^i(x)) )_{i=0}^{k-1}))(w), \ y \in (w - r_0, w + r_0),  \\ 
 &&   \iota_\cH(( h^{\sigma', w}_{i} )_{i=0}^{k-1})([w - r, w + r]) \Subset  \iota_\cH(( h^{\sigma, w}_{i} )_{i=0}^{k-1})((w - r, w + r)),  0 \leq \sigma < \sigma' \leq \sigma_0, 0 < r < r_0.
 \earyst
 \eenmt
\end{lemma}
\begin{proof}

This lemma is clear if  $(\cH, d_\cH, \iota_\cH)$ is in Example \ref{exam I}, since we can make perturbations using the projective action on the circle by $SL(2, \R)$.

Now we assume that $(\cH, d_\cH, \iota_\cH)$ is in Example \ref{exam II}. Let us denote $f = \Phi(H)$, and let $k$ be as in Definition \ref{def contractable}.

Fix an arbitrary $x \in X$. 
We define a function $c: X  \times \T \times \R \to \R$ by
\aryst
c(x, w, t) =  P(w) + H(x) + t + P(w + P(w) + H(x) + t) + H(g(x)).
\earyst
It is clear that $c$ is continuous.
By definition, we have
\ary \label{eq partialtat0is0}
 (f_{g(x)})_{-c(x, w, t) } \circ (f_x)_t(w)  =  w.
\eary
By straightforward computation, given any $w_0 \in \T$, we have
\ary \label{eq partialypartialtnot0}
\partial_t  \partial_w  \{ ( f_{g(x)})_{-c(x, w_0, t)} \circ (f_x)_t(w) \}|_{t = 0, w = w_0} = 
P''( f_x(w_0) )( 1 + P'(w_0)) \neq 0
\eary
as long as $P''( f_x(w_0) ) \neq 0$. In this case, the above term has the same sign as $P''( f_x(w_0) )$.

By compactness and our hypothesis that $H$ is contractable, there exists a finite open covering $\{  S_{\alpha} \times T_{\alpha} \}_{\alpha \in I}$ of $X \times \T$, such that for each $\alpha \in I$, there exists an integer $0 \leq m_{\alpha} \leq k - 2$ such that $P''( (f^{m_{\alpha} + 1})_{x}(w)) \neq 0$ for every $(x, w) \in S_{\alpha} \times T_{\alpha}$. 
Since $X$ is locally compact and Hausdorff, there exists a partition of unity $\{ \rho_{\alpha}  \}_{\alpha \in I}$ subordinated to $\{  S_{\alpha} \times T_{\alpha} \}_{\alpha \in I}$.
Moreover, again by compactness and by \eqref{eq partialypartialtnot0}, there exists some $\delta > 0$ such that
\ary
&&\inf_{t \in (0, \delta]} \inf_{\alpha \in I} \inf_{(x, w_0) \in S_{\alpha} \times T_{\alpha}} t^{-1} |   \partial_w  \{ (f_{g(x)})_{-c(x, w_0, t)} \circ (f_x)_t(w) \}|_{ w = (f^{m_{\alpha}})_{x}(w_0)} |  \nonumber \\
&\geq& \inf_{\alpha \in I} \inf_{(x, w_0) \in S_{\alpha} \times T_{\alpha}} | P''((f^{m_{\alpha} + 1})_{x}(w_0))(1 + P'((f^{m_{\alpha}})_{x}(w_0))) + o(1)  | > 0. \label{eq derivativelowerbound}
\eary

For any $0 \leq i \leq k -1$, we denote 
\aryst
e_{k, i} = (\delta_{0, i}, \delta_{1, i}, \cdots,  \delta_{k-1, i}) \in \R^{k}.
\earyst
We define a map $Q: X \times \T \times [0,1] \to \R^{k}$ by
\aryst
Q(x, w, t) = \sum_{\alpha \in I} \rho_{\alpha}(x, w) \sgn( P''((f^{m_{\alpha} + 1})_x(w)) )( \delta t e_{k, m_{\alpha}} - c(x, w, \delta t )  e_{k, m_{\alpha} + 1}).
\earyst
We may define 
\aryst
 E(\sigma, w, x, \check{H}) = (\check{H}(g^{i}(x)) )_{i=0}^{k - 1} -  Q(x, w, \sigma).
\earyst
We clearly have item (1). We can deduce item (2) from \eqref{eq partialtat0is0}.
We can deduce item (3) for $\check{H} = H$ by  a straightforward computation using  \eqref{eq partialypartialtnot0} and \eqref{eq derivativelowerbound}. Then we can verify item (3) for a general $\check{H} \in \cB_{\cH}(H, \epsilon)$ by continuity.
\end{proof}

From the above proof, we also have the following result, which will be used in the proof of Lemma \ref{lem getparabolicelement}.
\begin{lemma} \label{lem moveonefixone}
Given an arbitrary $H \in C^0(X, \cH)$ and denote $f = \Phi(H)$. Then
for any $x \in X$, $w_0, w_1 \in \T$ with $w_0 \neq w_1$, for any $\epsilon > 0$, there exists $(p_0, p_1) \in \cB(H(x), \epsilon) \times \cB(H(g(x)), \epsilon)$ such that $\iota_{\cH}(p_1) \iota_{\cH}(p_0) (w_0) = (f^2)_x(w_0)$ and  $\iota_{\cH}(p_1) \iota_{\cH}(p_0) (w_1) \neq (f^2)_x(w_1)$.
\end{lemma}
\begin{proof}
The statement is clear if $(\cH, d_\cH, \iota_\cH)$ is in Example \ref{exam I}. 

Now we assume that $(\cH, d_\cH, \iota_\cH)$ is in Example \ref{exam II}. 
Let $t$ be a constant close to $0$ to be determined. We set 
\aryst
p_0 = H(x) + t, \ p_1 = H(g(x)) - c(x, w_0, t).
\earyst
Following the computations in Lemma \ref{lem arnoldfamilyismovable}, we see that 
\aryst
\iota_{\cH}(p_1) \iota_{\cH}(p_0) (w) = w + P(w) + P( w + P(w) + H(x) + t)  -  P(w_0)  -  P( w_0 + P(w_0) + H(x) + t).
\earyst
Then it is clear that $\iota_{\cH}(p_1) \iota_{\cH}(p_0) (w_0) = w_0$.
Since $w_0 \neq w_1$, we have $ w_1 + P(w_1) + H(x)   \neq w_0 + P(w_0) + H(x)$. Then since $P$ has no smaller period, $\iota_{\cH}(p_1) \iota_{\cH}(p_0) (w_1)$ is a non-constant real analytic function of $t$. Hence there exists $t$ arbitrarily close to $0$ such that  $\iota_{\cH}(p_1) \iota_{\cH}(p_0) (w_1) \neq w_1$. This concludes the proof.
\end{proof}

We first reduce Theorem \ref{Theorem 30} to the following proposition by a standard argument.

\begin{proposition}\label{PROPOSITION 36} 
For  any contractable $H \in C^0(X, \cH) \setminus \overline{ \mathcal{ML}(\cH) }$ and any $\epsilon > 0$, there exists a contractable $H' \in C^0(X, \cH) \setminus  \overline{ \mathcal{ML}(\cH) } $ such that 
$D_{\cal H}(H, H') < \epsilon$ and $\vert L_+(H')\vert, \vert L_-( H' )\vert < \epsilon$. 
 \end{proposition}
 
 \begin{remark}
 If $ C^0(X, \cH)  = \overline{ \mathcal{ML}(\cH) }$, then the condition of Propostion \ref{PROPOSITION 36} is void. In this case, the conclusion of  Theorem \ref{Theorem 30} is already satisfied.
 \end{remark}

We can easily deduce Theorem \ref{Theorem 30} from Proposition \ref{PROPOSITION 36}.
\begin{proof}[Proof of Theorem \ref{Theorem 30}]  
Let us assume that $C^0(X, \cH) \neq  \overline{\mathcal{ML}(\cH)}$.
 For any $\epsilon > 0$, we denote
$$\mathcal{U}_\epsilon := \{ H \in C^0(X, \cH) \setminus \overline{ \mathcal{ML}(\cH) } \mid
\vert L_+( H )\vert, \vert L_-(H)\vert<\epsilon\}.$$ 
Given an arbitrary $\epsilon > 0$. By the upper- , resp. lower-, semicontinuity of $L_+$, resp. $L_-$, we see that $\mathcal{U}_\epsilon$ is open. By Proposition \ref{PROPOSITION 36} and Lemma \ref{lem densityofcontractable}, $\mathcal{U}_\epsilon$ 
is dense. Then the set $\mathcal{U}_0 := \cap_{n\geq 1} \mathcal{U}_{\frac{1}{n}}$
is a residual subset of $C^0(X, \cH) \setminus \overline{\mathcal{ML}(\cH)}$. By definition, every $H \in \mathcal{U}_0$ satisfies $L_+( H ) = L_- ( H  ) = 0$.
\end{proof}

We will deduce Proposition \ref{PROPOSITION 36} from the following slightly more technical proposition.
  
\begin{proposition}\label{proposition 37}
For any contractable $H \in  C^0(X, \cH) \setminus \overline{ \mathcal{ML}(\cH) }$ such that $L_+(H)>L_-(H)$, for any $\epsilon > 0$,  there exists a  contractable  $H' \in  C^0(X, \cH) \setminus \overline{ \mathcal{ML}(\cH) }$
such that $D_{\cal H}(H, H') < \epsilon$ and
$$ \max(-L_-( H' ), L_+( H' )) < \max(-L_-( H ), L_+( H ))(1-  10^{- 6}(\frac{ L_+( H ) - L_-( H )}{ \log\Vert \bPhi(H) \Vert_{C^{0,1}} + 3 })^2). $$
\end{proposition}

\begin{proof}[Proof of Proposition \ref{PROPOSITION 36} assuming Proposition \ref{proposition 37}] By Lemma \ref{LEMMA 65}, $L_-( H  ) \leq 0 \leq L_+(  H  )$. Without loss of generality, we can assume $L_+(  H  )-L_-(  H  ) \geq \epsilon$, for otherwise we can let $H' = H$.  Without loss of generality, let us assume that 
\aryst
\cB_{\cal H}(H, 2\epsilon) \subset C^0(X, \cH) \setminus \overline{ \mathcal{ML}(\cH) }.
\earyst
Denote
\aryst
L = \log (\Vert  \bPhi(H)  \Vert_{C^{0,1}} + 1) + 3.
\earyst

Define $H_0 = H$. Assume that for some integer $n\geq 0$, we have constructed some contractable $H_n \in C^0(X, \cH) \setminus \overline{ \mathcal{ML}(\cH) }$ so that $D_{\cal H}(H_n,  H)\leq (2^{-1}-  2^{-n-1})\epsilon$ and $L_+(  H_n  ) -  L_-(  H_n ) \geq\epsilon$. Without loss of generality, we may assume that $\epsilon$ is sufficiently small so that we have by Remark \ref{rem comparenorms} that 
\aryst
\| \Phi(H_n) \|_{C^{0,1}}  \leq  \| \Phi(H) \|_{C^{0,1}} + C D_{\cal H}(H_n,  H)   \leq   \| \Phi(H) \|_{C^{0,1}} + 1.
\earyst
Then by Proposition \ref{proposition 37}, we can find a contractable $H_{n+1}\in C^0(X, \cH) \setminus \overline{ \mathcal{ML}(\cH) }$ so that 
$D_{\cal H}(H_{n+1}, H_n)\leq 2^{-n-2}\epsilon$ and  
$$ \max(-L_-( H_{n+1}), L_+( H_{n+1})) < \max(-L_-( H_n), L_+( H_n))(1- 10^{-6}(\frac{\epsilon}{ L})^2 ).$$
Notice that we have $D_{\cal H}(H_{n+1}, H)\leq (2^{-1}-  2^{-n-1})\epsilon + 2^{-n-2}\epsilon \leq (2^{-1}-  2^{-n-2})\epsilon$.
Then for some integer $m > 0$, we would have $D_{\cal H}(H_m,  H) < \epsilon$ and 
$L_+(H_m) -  L_-(H_m) < \epsilon$. We let $H' =  H_m$ and this concludes the proof.
\end{proof}

The rest of this section is dedicated to the proof of Proposition \ref{proposition 37}. 


%
%

We have the following important lemma.

\begin{lemma}\label{lem mainpropertymovableset}
For any contractable $H \in  C^0(X, \cH) \setminus \overline{ \mathcal{ML}(\cH) } $ such that $L_+( H ) >L_-( H )$, for any $\epsilon  > 0$, there exists $N_4 =  N_4(H,\epsilon) > 0$ such that the following is true. For any $x\in X$, any integer $N \geq  N_4$, there exists 
$(p_0,\cdots, p_{N-1})\in \cH^N$ such that
\begin{enumerate}
\item $d_{\cH}(p_i, H(g^i(x)))<\epsilon$ for every $0\leq i\leq N-1$,
\item  we have 
\aryst
\max(-{\bf L}_-, {\bf L}_+)((p_i)_{i=0}^{N-1})  <  (1-\lambda_0)\max (-L_-(H),L_+(H))
\earyst
where
\ary \label{eq deflambda0}
\lambda_0 = 10^{- 5}(\frac{L_+( H )-L_-( H )}{   \log \|  \bPhi(H) \|_{C^{0,1}}  + 3 } )^2.
\eary
\end{enumerate}
\end{lemma}

The proof of Lemma \ref{lem mainpropertymovableset} is technical and will be deferred to  Section \ref{sec Perturbation lemmata}.

 We are now ready to state the proof of  Proposition \ref{proposition 37}.
\begin{proof}[Proof of Proposition \ref{proposition 37}] 
Let us fix some contractable $H \in  C^0(X, \cH) \setminus \overline{ \mathcal{ML}(\cH) }$ such that $L_+( H )>L_-( H) $. Assume to the contrary that the conclusion of Proposition \ref{proposition 37} is false for $H$.

By Lemma \ref{LEMMA limitLE}, we set 
\aryst 
N'_0 = N'_0(H, 1),
\earyst
then for any $x \in X$, for any $n > N'_0$, we have that 
\aryst
 {\bf L}_+( (H(g^{i}(x)))_{i=0}^{n-1} ) \leq  L_+(H) + 1, \  {\bf L}_-( (H(g^{i}(x)))_{i=0}^{n-1}  ) \geq  L_-(H) - 1.
\earyst

We set
\aryst
N_4 =  N_4(H, \frac{1}{4}\epsilon), 
\earyst
where the function $N_4$ is given by Lemma \ref{lem mainpropertymovableset}.
We fix an arbitrary integer
\ary
 N >  N_4 + N'_0.  \label{8.6.25} 
 \eary
Denote by $\nu$ the unique $g$-invariant measure. Under the hypothesis of $g$, we can choose a subset $B\subset X$ by \cite[Lemma 6]{AB} such that the return time from $B$ to itself via $g$ equals to either $N$ or $N + 1$, and $\nu(\partial B) = 0$. We fix such $B$ from now on.

By reducing the size of $\epsilon$ if necessary, we may assume that for any $H' \in \cB_{\cH}(H, \epsilon)$, we have, for $n \in \{ N, N+1 \}$, that
\ary \label{eq L+L-bL+bL-} 
 {\bf L}_+( (H'(g^{i}(x)))_{i=0}^{n-1} ) \leq  L_+(H) + 2, \  {\bf L}_-( (H'(g^{i}(x)))_{i=0}^{n-1}  ) \geq  L_-(H) -  2. 
\eary
Let $\delta > 0$ be small constant such that
$$\sup_{x\in X}\sup_{x'\in B(x,\delta), 0\leq i\leq N} d_{\cal H}(H(g^i(x')), H(g^i(x)) )<\frac{1}{2}\epsilon. $$
Cover the closure of $B$ by open sets $W_1,\cdots,W_{k_1}$ with diameters less than $\delta/2$. By  \cite[Lemma 3]{AB}, we can choose $W_i$ so that $\nu(\partial W_i) = 0$ for all $1\leq  i \leq k_1$. Let $U_i = W_i \setminus \cup_{j<i} W_j$. After discarding those $U_k$ which are disjoint from $B$, and rearranging the indexes, we can assume that for some integer $k_0 > 0$, for each $1\leq k\leq k_0$, $B\cap U_k\not=\emptyset$; and the union of $U_k$ over $1\leq k\leq k_0$ covers $B$. 

By our choice, we have $B=B_N\cup B_{N+1}$ where $B_l$ denotes the set of points in $B$ whose first return time to $B$ equals to $l$. Let
$$ V_0 = \bigcup^{N+1}_{l=N}\bigcup^{l-1}_{i=0}\partial(B_l\cap U_i).$$
Then $\nu(V_0 )=0.$

Denote by $\Gamma$ the set of $(k,l)$ such that $1\leq  k\leq k_0, l\in\{N,N+1\}$ and $B_l \cap U_k\not=\emptyset$. For each $(k,l)\in\Gamma$, we choose a point $w_{k,l}\in B_l\cap U_k \setminus V_0$.

Fix an arbitrary $(k, l) \in \Gamma$. Note that $l \geq N > N_4$. By Lemma \ref{lem mainpropertymovableset} for $(\epsilon/2, H, w_{k, l}, l)$ in place of $(\epsilon, H, x, N)$, we obtain  ${\bf p} = (p^{k,l}_{0},\cdots, p^{k,l}_{l-1})\in \cal H^l$ such that
\enmt
\item we have $d_{\cal H}(p^{k,l}_{i},H(g^i(w_{k,l})))<\frac{1}{2}\epsilon$ for any $0\leq  i\leq l-1$,
\item we have
\ary \label{eq upperboundLE}
\max(-\bL_-({\bf p}), \bL_+({\bf p})) < (1 - \lambda_0) \max(-L_-(H),L_+(H)). 
\eary
\eenmt

 Now let $\eta  > 0$ be a sufficiently small constant to be determined later. By the unique ergodicity of $g$, there exist an open set $V\supset V_0 $ and an integer $n_0 >0$ such that
  \begin{equation}
\frac{1}{n}\left| \{ j\mid 0\leq j\leq n-1, g^j(x)\in V \} \right|<\frac{\eta}{N+1}, \forall x \in X, \forall n\geq n_0.
 \label{8.6.28} 
 \end{equation}
 Moreover, we can assume that $w_{k,l}\not\in V$ for any $(k, l) \in \Gamma$.
 
We define a map $H' \in C^0(X, {\cal H})$ by using Proposition \ref{prop-tietze} so that
$D_{\cal H}(H, H') < \epsilon$, and 
 for any $(k,l)\in\Gamma$, any $0\leq i\leq l-1$, we have $H'(g^i(x)) = p^{k,l}_{i}$ for all $x \in B_l\cap U_k \setminus V$.

 By letting $\eta > 0$ be sufficiently small depending only on $H, \epsilon$  (this can be realized by choosing $V$ of sufficiently small measure, and by letting $n_0$ be sufficiently big),  we can ensure  by \eqref{eq L+L-bL+bL-}, \eqref{eq upperboundLE} and \eqref{8.6.28} that
 \aryst
 L_+(H') &\leq& \max(-L_-(H),L_+(H))(1- \lambda_0) (1 - \eta \frac{N+1}{N})  \\
 && + (L_+(H) + 2) \eta \frac{N+1}{N} \\
 &\leq& \max(-L_-(H),L_+(H))(1-\frac{1}{2}\lambda_0).
 \earyst
 By a similar argument, we obtain an analogous bound for $ - L_-(H')$.
 Consequently, we have
 $$ \max(-L_-(H'),L_+(H')) < \max(-L_-(H),L_+(H))(1-\frac{1}{2}\lambda_0).$$
We see that $H'$ satisfies the conclusion of Proposition \ref{proposition 37}. 
\end{proof}

\section{Negative Lyapunov exponent} \label{sec Negative Lyapunov exponent}
 Recall that for  a lift $F$  of a $g$-forced circle homeomorphism $f$, for every integer $N > 0$ and every $\kappa > 0$ we have defined  $\Omega_N(F, \kappa)$ in \eqref{eq omegaNFkappa}.

 Given a contractable $H \in C^0(X, {\cal H})$, we let $\epsilon > 0$ be a sufficiently small constant, and we let the integer $k > 0$ and the continuous map $E : [0, 1] \times \T \times X \times \cB_{\cal H}(H, \epsilon) \to \cH^k$ be given by Lemma \ref{lem arnoldfamilyismovable}.

\begin{proposition} \label{prop 38} 
Given a contractable $H \in  C^0(X, \cH) \setminus \overline{\mathcal{ML}(\cH)}$ such that   
$L_+(H) = L_-(H) = 0$, and let $F$ be a lift of $\Phi(H)$.
 For any $\epsilon > 0$, there exist $\kappa_3 = \kappa_3(H,\epsilon)\in (0, \frac{1}{2}), N_5 = N_5(H,\epsilon)>0$ such that for any integer $N\geq N_5$ there exists  $r_3 = r_3(H, N, \epsilon)>0$  such that for any $\check{H} \in \cB_{\cH}(H,  \kappa_3)$, for any $\bar{r}\in (0,r_3)$, there exists 
a continuous function $\Psi_N:  \Omega_N(F,\kappa_3)  \rightarrow  \cal H^N$ such that the following is true:
 For any $(x, y, z) \in \Omega_N(F, \kappa_3)$, let $\Psi_N(x, y, z) = (p_0, \cdots, p_{N-1})$, and let $P_i$ be the unique lift of $\iota_{\cH}(p_i)$ that is close to $F_{g^{i}(x)}$, then we have
\begin{enumerate}
\item $d_{\cH}(p_i,\check{H}(g^i(x)) ) < 2\epsilon$ for every $0\leq i\leq N-1$, 
\item $P_{N-1} \circ \cdots \circ P_0([y-\bar{r}, y+\bar{r}])\subset [z-\frac{1}{10}\bar{r}, z+\frac{1}{10}\bar{r}],$
\item if 
$(\check{F}^{N-1})_{x}(y) = z$ and $(\check{F}^{N-1})_{x}([y-\bar{r}, y+\bar{r}])\subset  [z-\frac{1}{10}\bar{r}, z+\frac{1}{10}\bar{r}]$, where $\check{F}$ is the unique lift of $\Phi(\check{H})$ that is close to $F$, then $p_i=\check{H}(g^i(x))$ for every $0\leq i\leq N-1$.
\end{enumerate}
\end{proposition}

\begin{proof}

Let us denote for simplicity $f = \bPhi(H)$ and then $F$ is a lift of $f$.
Without loss of generality, we assume $\epsilon\in (0,1)$ is sufficiently small to apply   Lemma \ref{lem arnoldfamilyismovable}.

We fix a small constant  $\sigma_0 = \sigma_0(f, \epsilon)  > 0$ such that for every $\check{H} \in \cB_{\cH}(H, \epsilon)$,  $x \in X$ and $w \in \T$ we have
\ary \label{eq Ehasnarrowrange}
E(\sigma_0, w, x, \check{H}) \in \prod_{i=0}^{k-1} \cB_{\cal H}( \check{H}(g^{i}(x)), \epsilon),
\eary
where the map $E$ is given by  Lemma \ref{lem arnoldfamilyismovable}.
We let $r_0, \epsilon_2$ be given by Lemma \ref{lem arnoldfamilyismovable}(3) for $\sigma_0$.

%

Since $L_+(f) = L_-(f) = 0$, there exists $N' =N'(H,\epsilon)>0 $ such that
\aryst
 \sup_{x\in X, w \in \mathbb{T}} |\log D(f^n)_x(w)|< \frac{n\epsilon_2}{4k}, \  \forall n>N'.
 \earyst
Then we choose $\kappa' = \kappa'(H,\epsilon) > 0$ to be sufficiently small so that for any $H' \in \cB_{\cH}(H, \kappa')$, by denoting $f' = \bPhi(H')$, we have 
 \begin{equation}
 \sup_{x\in X, w \in \mathbb{T}} |\log D(f')^n_x(w)|< \frac{n\epsilon_2}{2k},  \ \forall n>N'.
\label{8.7.2} 
 \end{equation}

We let $\epsilon_0 = \epsilon_0(H, \epsilon) \in (0, \epsilon)$ be given by Lemma \ref{lem cancellation}. 
We let $\kappa'' = \kappa_1(H,\epsilon_0)$ be given by Corollary \ref{LEMMA 56}. By Lemma \ref{LEMMA 55}, there exists some $N'' = N''(H,\epsilon) > 0$ such that for any $n > N''$, we have
 \begin{equation}
 \sup_{x\in X, y\in\mathbb{R}}|(F^n)_x(y)-y-n\rho(F)|<
 \frac{\frac{1}{60}\epsilon_2}{k( \log(\Vert f\Vert_{C^{0,1}}+1)+1) }n\kappa''. \label{ (8.7.3)} 
 \end{equation}
We choose $\kappa''' > 0$ to be sufficiently small, depending only on $H, \kappa''$ and $\epsilon_2$, so that for any $H' \in \cB_{\cal H}(H,  \kappa''')$, for any $n > N''$, we have 
 \begin{equation}
 \sup_{x\in X, y\in\mathbb{R}}|(F')^n_x(y)- (F^n)_x(y) |<
 \frac{\frac{1}{20}\epsilon_2}{k(\log(\Vert f\Vert_{C^{0,1}}+1)+1)}n\kappa'' \label{8.7.4} 
 \end{equation}
 where $F'$ denotes the lift of $f' = \bPhi(H')$ that is close to $F$.
 
 We define
 \begin{align} 
 \kappa_3& = \frac{1}{2}\min(\epsilon_0, \kappa',\kappa''', \frac{\frac{1}{100}\epsilon_2}{k ( \log(\Vert f\Vert_{C^{0,1}}+1)+1 )}\kappa''),\label{8.7.5} \\
 N_5 &= 2N' + 2N'' + 100 \frac{k}{\epsilon_2}(\log(\Vert f\Vert_{C^{0,1}}+1) + 1)N_1(H,\epsilon_0) + 100 \frac{k}{\epsilon_2}.  \label{8.7.6}
 \end{align}
 
 Let $N > N_5$ and $(x, y, z)\in  \Omega_N(F, \kappa_3)$. We define
 \ary
 \label{(8.7.7)}  \bar{N} &=& \lceil \frac{1}{k}( 1 - \frac{\epsilon_2}{10 k (\log(\Vert f\Vert_{C^{0,1}} + 1))}  )N  \rceil,  \\
r_3 &=& (\Vert f\Vert_{C^{0,1}}+1)^{-N} r_0.
\label{8.7.10}
\eary
 Let $\check{H}$ be as in the proposition, and let $\check{F}$ be the unique lift of $\check{f} = \bPhi(\check{H})$ that is close to $F$. For any $\sigma\in  [0, \sigma_0]$, we define
\aryst
( v^\sigma_{ik + j})_{j=0}^{k-1} = E( \sigma, ( \check{f}^{ik})_x(y\mod 1), g^{ik}(x), \check{H} ),  \ \forall 0 \leq i \leq \bar{N}-1.
\label{8.7.8}
\earyst
By Lemma \ref{lem arnoldfamilyismovable} (2), (3) and \eqref{eq Ehasnarrowrange}, for every $\sigma\in [0, \sigma_0]$ we have that
\ary \label{eq vsigmalclose}
d_{\cH}(v^{\sigma}_l, \check{H}(g^{l}(x))) < \epsilon, \ \forall 0 \leq l \leq \bar{N}k-1
\eary
and
\ary
V^\sigma_{ik-1}\cdots V^\sigma_0(y)= (\check{F}^{ik})_x(y), \ \forall 1\leq i\leq \bar{N}  \label{8.7.9}
\eary
where $V^\sigma_j$ is the unique lift of $\iota_{\cH}(v^\sigma_j)$ that is close to $\check{F}_{g^j(x)}$. 

We have the following.
\begin{claim}\label{claim 1}
For any $0\leq i\leq \bar{N} - 1$, we have
 \aryst
 D(V^{\sigma_0}_{(i+1)k-1}\cdots V^{\sigma_0}_{ik})(y') < e^{- \epsilon_2} D((\check{F}^k)_x)( (\check{F}^{ik})_x(y)), \forall y' \in ( (\check{F}^{ik})_x(y) -  r_0,  (\check{F}^{ik})_x(y) +  r_0),
 \earyst
 and for any $r''\in (0,  r_0 )$, and any $0\leq \sigma_1 <\sigma_2 \leq\sigma_0$, we have
\aryst
V^{\sigma_2}_{(i+1)k-1}\cdots V^{\sigma_2}_{ik}( (\check{F}^{ik})_x(y)+[-r'',r''])\Subset V^{\sigma_1}_{(i+1)k-1}\cdots V^{\sigma_1}_{ik}( (\check{F}^{ik})_x(y)+(-r'',r'')).
 \earyst
In particular, for any $\bar{r}\in(0, r_3)$, we have
\aryst
&&V^{\sigma_2}_{\bar{N} k -1}\cdots V_0^{\sigma_2}([y-\bar{r},y+ \bar{r}])
\Subset  V^{\sigma_1}_{\bar{N}k - 1}\cdots V_0^{\sigma_1}([y-\bar{r}, y+\bar{r}]). \nonumber
\earyst
\end{claim}
\begin{proof} 
The inequality and the first inclusion follow immediately from Lemma \ref{lem arnoldfamilyismovable}(3). The last statement follows from \eqref{8.7.10} by repeatedly applying the first statement.
\end{proof}

By (\ref{8.7.4}), $(x, y, z)\in \Omega_N(F, \kappa_3)$ and (\ref{8.7.5}), we have
\[
\begin{array}{ccc}
|  (\check{F}^N)_x(y)-z| &\leq & | ( F^N)_x(y)-z| + | (\check{F}^N)_x(y) - (F^N)_x(y)|\\
&\leq & N(  \kappa_3 + \frac{\frac{1}{20}\epsilon_2}{k (\log(\Vert f\Vert_{C^{0,1}}+1)+1)}\kappa'')<(N-\bar{N} k)\kappa''.
\end{array}
\]
Hence we have $(g^{\bar{N} k}(x), (\check{F}^{\bar{N} k})_x(y),z) \in \Omega_{N-\bar{N} k}(F, \kappa'')$.
Moreover, by (\ref{8.7.6}) and $N>N_5$, we have $N-\bar{N}k >N_1(f,\epsilon_0)$. Then we can apply  Lemma \ref{lem cancellation} to define
$$ (u_{\bar{N}  k },\cdots, u_{N-1}) =  \Phi^{\check{H}}_{N-\bar{N} k}(g^{\bar{N} k}(x), (\check{F}^{\bar{N} k})_x(y),z).$$ 
We have 
\ary \label{eq gicheckficlose}
d_{\cal H}(u_i, \check{H}(g^i(x))) < 2\epsilon, \ \forall \bar{N} k \leq i\leq N-1.
\eary
For every $\bar{N} k \leq i\leq N-1$, let us denote  by $U_i$ the unique lift of $\iota_\cH(u_i)$ that is close to $\check{F}_{g^{i}(x)}$.

\begin{claim}\label{claim 2}
 For any $\bar{r}\in(0, r_3]$, we have
$$U_{N-1}\cdots U_{\bar{N} k}V^{\sigma_0}_{\bar{N} k - 1}\cdots V_0^{\sigma_0}([y-\bar{r},y+ \bar{r}])\subset [z-\frac{1}{10}\bar{r}, z+\frac{1}{10}\bar{r}].$$
\end{claim}
\begin{proof}
By \eqref{8.7.2} and $\kappa_3 < \kappa'$, we see that
\ary \label{eq concatenate1steqinClaim52}
 \sup_{x \in X, y \in \R} |\log D (\check{F}^n)_x(y)| < \frac{n\epsilon_2}{2k}, \ \forall n > N'.
\eary
By \eqref{8.7.10}, $\bar r \leq r_3$, \eqref{eq concatenate1steqinClaim52} and by repeatedly applying Claim \ref{claim 1}, we obtain
\aryst
D(V^{\sigma_0}_{\bar{N} k -1}\cdots V_0^{\sigma_0})(y')< e^{-\bar{N}\epsilon_2/2},\ \forall y'\in (y-\bar{r}, y+\bar{r}).
\earyst
Thus 
\ary
V^{\sigma_0}_{\bar{N}k -1}\cdots V_0^{\sigma_0}((y-\bar{r}, y+\bar{r}))\subset 
V^{\sigma_0}_{\bar{N}k -1}\cdots V_0^{\sigma_0}(y)+e^{-\bar{N} \epsilon_2/2}(-\bar{r},\bar{r}). \label{8.7.11}
\eary
By  \eqref{8.7.6}, \eqref{(8.7.7)}, \eqref{eq gicheckficlose} and (\ref{8.7.11}), we have
\ary
&&U_{N-1}\cdots U_{\bar{N} k} V^{\sigma_0}_{\bar{N} k - 1}\cdots V_0^{\sigma_0}([y-\bar{r},y+\bar{r}]) \label{8.7.12}\\
&\subset & z+(\Vert f\Vert_{C^{0,1}}+1)^{N-\bar{N}k}e^{-\bar{N}\epsilon_2/2}[-\bar{r},\bar{r}]\subset [z-\frac{1}{10}\bar{r}, z+\frac{1}{10}\bar{r}], \nonumber
\eary
since we have
\aryst
(N-\bar{N}k) \log (\Vert f\Vert_{C^{0,1}}+1) - \bar{N}\epsilon_2/2 \leq - \bar{N}\epsilon_2/4 \leq - \log 10.
\earyst
\end{proof}

Let $\bar{r}\in(0, r_3)$ be given by the proposition. We define
$$\sigma_1 = \inf\{\sigma\in [0,\sigma_0]\mid U_{N-1}\cdots U_{\bar{N} k } V^\sigma_{\bar{N} k -1} \cdots V^\sigma_0([y-\bar{r}, y+\bar{r}]) \subset [z-\frac{1}{10}\bar{r}, z+\frac{1}{10}\bar{r}]\}.$$
  By Claim \ref{claim 2}, we see that $\sigma_1$ is well-defined.   By the last inclusion in Claim \ref{claim 1}, $\sigma_1$ depends continuously on $(x, y, z)$. Let us define 
  \aryst
  p_i = \begin{cases}  
  v^{\sigma_1}_{i}, & 0 \leq i \leq \bar{N}k  - 1, \\
  u_i, & \bar{N} k  \leq i \leq N-1.
  \end{cases}
  \earyst
   Conclusion (1)-(3) then follow from the construction. 
\end{proof} 

\section{Density of mode-locking} \label{sec Density of mode-locking}
In this section, we will first give the proof of Theorem \ref{thm:29'}, and then deduce Theorem \ref{thm:29} as a corollary.
\begin{proof}[Proof of Theorem \ref{thm:29'}] 
Let us assume to the contrary that 
\aryst
C^0(X, \cH) \neq \overline{\mathcal{ML}(\cH)}.
\earyst
Given an arbitrary  $H \in C^0(X, \cH) \setminus \overline{\mathcal{ML}(\cH)}$ and an arbitrary constant $\epsilon\in  (0, 1)$. By Lemma \ref{lem densityofcontractable} and Theorem \ref{Theorem 30}, up to replacing $H$ by an arbitrarily close element,  we may assume without loss of generality that $H$ is contractable and  $L_+( H )= L_-( H )=0$.   It remains to show that there exists $H' \in  \mathcal{ML}(\cH)$ such that $D_{\cal H}(H, H')<\epsilon$, as  this would be a contradiction.
In order to simplify the notation, let us denote $f = \bPhi(H)$, and let $F$ denote a lift of $f$.

Let integer $d > 0$ be given by Lemma \ref{LEMMA 58}.
Let $\kappa_3$ be given by Proposition \ref{prop 38}. Without loss of generality, we may assume that $\kappa_3(f, \delta)$ is monotonically increasing in $\delta$.
We inductively define positive constants $0 < \epsilon_{-1} < \epsilon_{0} < \cdots < \epsilon_{d}$ by the following formula,
\aryst
 \epsilon_d=\frac{\epsilon}{4(d+1)},   \ 
\epsilon_{d-k}=\min(\frac{1}{2}\epsilon_{d-k+1},\frac{1}{2(d+1)}\kappa_3(f,\epsilon_{d-k+1})),  \ \forall 1\leq k\leq d+1.
\earyst
Then we have
\ary
2(\epsilon_0+\cdots+\epsilon_d)<\epsilon, \   2(\epsilon_0+\cdots+\epsilon_k)< \kappa_3(f,\epsilon_{k+1}),\forall 0\leq k\leq d-1. \label{8.8.1}
\eary

We set $\kappa' = \inf_{1\leq i\leq d} \kappa_3(f, \epsilon_i)$. 
Recall that by our hypothesis, $g$ is uniquely ergodic. Let us denote by $\nu$ the unique $g$-invariant measure on $X$.
By our hypothesis that $G(g)$ (see Definition \ref{8.2.1}) is dense in $\mathbb{R}$, we can choose a constant
\begin{equation}
\rho'\in(\rho(F)-\frac{1}{4}\kappa',\rho(F)+\frac{1}{4}\kappa')\cap G(g). 
\label{8.8.2}
\end{equation}
By Definition \ref{8.2.1}, there exist continuous maps $\phi : X\rightarrow \mathbb{R}$ and $\psi: X \rightarrow \mathbb{R}/\mathbb{Z}$ so that $\rho' = \int \phi d\nu$ and $\psi(g(x))-\psi(x)=\phi(x) \mod 1.$

Since $g$ is uniquely ergodic, we can choose $n_0 = n_0(f,\epsilon) > 0$ to be a large integer so that
\ary
n_0>N_5(H,\epsilon_i), \ \forall -1\leq i\leq d, \label{8.8.3} \ \mbox{(see Proposition \ref{prop 38} for $N_5$)}
\eary
and for any integer $n > n_0$, we have
\ary
\sup\limits_{x\in X}\vert \Sigma^{n-1}_{i=0} \phi(g^i(x))-n\rho' \vert < \frac{1}{4}n\kappa',  \label{8.8.4} \
\sup\limits_{x\in X, y \in \R} \vert  (F^{n})_{x}( y ) -  y  - n \rho(F) \vert < \frac{1}{4}n\kappa'.
\eary

We choose an open set $U \subset X$ such that $\phi|_U$ admits a continuous lift $\hat{\psi}: U \rightarrow \R$. Namely, $\hat{\psi}$ is continuous and $\psi(x) = \hat{\psi}(x) \mod 1$ for any $x \in U$.

By Lemma \ref{LEMMA 58} and by enlarging $n_0$ if necessary, we can choose a compact set $K \subset U$ that is $d$-mild, $n_0$-good and $M$-spanning for some $M > 0$.

 We choose an arbitrary $\bar{r}\in (0, \frac{1}{4})$ such that
\begin{equation} 
\bar{r}<r_3(H,l, \epsilon), \ \forall n_0< l\leq M, \label{8.8.5}
\end{equation}
where $r_3$ is given by Proposition  \ref{prop 38}.
 
Let $\{ K^i \}^d_{i=-1}, \{Z^i \}^d_{i=0}$ and $l: X \to \Z_+$  be defined as in Section \ref{secKZ}, associated to $K$. We will define a sequence of  $H^{(i)} \in C^0(X, \cH), 1 \leq  i\leq d $ by induction.
  
  We define $H^{(-1)} = H$, $f^{(-1)} = f = \bPhi(H^{(-1)})$ and $F^{(-1)} = F.$ Assume that we have defined $H^{(k)}$ for some  $-1 \leq k \leq d-1$ such that, let $F^{(k)}$ be the lift of $f^{(k)} = \bPhi(H^{(k)})$ that is close to $F^{(k-1)}$ if $k \geq 0$, then
  \begin{itemize}
  	\item[$(f1)$] $D_{\cal H}(H^{(k)}, H)\leq 2(\epsilon_0+\cdots+\epsilon_k),$
	\item[$(f2)$] for any $x\in K^k$, we have $((F^{(k)})^{l(x)})_x(\hat{\psi}(x))=\hat{\psi}(g^{l(x)}(x))+\Sigma_{j=0}^{l(x)-1}\phi(g^j(x)),$
	\item [$(f3)$] for any $x\in K^k$, we have $((F^{(k)})^{l(x)})_x(\hat{\psi}(x)+[-\bar{r},\bar{r}])\subset \hat{\psi}(g^{l(x)}(x))+\Sigma_{j=0}^{l(x)-1}\phi(g^j(x))+[-\frac{1}{10}\bar{r},\frac{1}{10}\bar{r}].$
  \end{itemize}
  Note that the above properties are true for $k =  -1$ simply because $K^{-1} = \emptyset$ by Lemma \ref{LEMMA 59}(4).
  
  For each $-1\leq j\leq d$, we let
  $$W^j = \bigcup_{x\in K^j} \bigcup_{0\leq i < l(x)} \{g^i(x)\}.$$ 
  Since $K^d = K$ is $M$-spanning, we have $W^d = X$. Recall that we have the following lemma.
  \begin{lemma}[Lemma 10 in \cite{Zha}]\label{LEMMA 62}
Given an integer $0 \leq j \leq d$, let $\{ x_n \}_{n \geq 0 }$ be a sequence of points in $K^j$ converging to $x'$, and let $ \{l_n \in [0, l(x_n)\}_{n \geq 0}$ be a sequence of integers converging to $l'$.
Then after passing to a subsequence, we have exactly one of the following possibilities:
\enmt
\item[either (1)] $x' \in Z^{j}$ and $0 \leq l' < l(x')$,
\item[or (2)]  $x' \in K^{j-1}$, and there exist a unique $x'' \in K^{j-1}$ and a unique $0 \leq l'' < l(x'')$ such that $g^{l'}(x') = g^{l''}(x'') \in W^{j-1}$.
\eenmt
In particular, $W^j$ is closed.
  \end{lemma}
     
  By (\ref{8.8.1}) and $(f1)$, we have $D_{\cal H}(H^{(k)}, H) < \kappa_3(f,\epsilon_{k+1})$. 
  By  \eqref{8.8.3} and \eqref{8.8.5}, we can apply Proposition \ref{prop 38} to $(\epsilon_{k+1}, l, H, H^{(k)},\bar{r})$ in place of $(\epsilon, N, H, \check{H}, \bar{r})$ to define $\Psi_l$  for all $n_0 <l\leq   M$.

We define a continuous map $\tilde{H} : W^{k+1} \to \cH$ such that $\tilde{H}(x) = H^{(k)}(x)$ for every $x \in W^k$ in the following way. Let
  \begin{equation}
\tilde{H}(x) = H^{(k)}(x),  \ \forall x\in W^k.  \label{8.8.6}
\end{equation}
For any $x\in Z^{k+1}$, we have $n_0 <l(x)\leq M$. By (\ref{8.8.2}) and (\ref{8.8.4}), we have
\aryst
 && \vert  \hat{\psi}(x)  +  \sum\limits_{i=0}^{l(x)-1} \phi(g^i(x)) -  (F^{l(x)})_{x}( \hat{\psi}(x) ) \vert \\
 &\leq& 
 \vert  \sum\limits_{i=0}^{l(x)-1} \phi(g^i(x))-l(x)\rho(F)\vert  + \vert  (F^{l(x)})_{x}( \hat{\psi}(x) ) -  \hat{\psi}(x)  - l(x)\rho(F)\vert \\
  &<& 3 \times \frac{1}{4}l(x)\kappa'  < l(x)\kappa_3(f,\epsilon_{k+1}).
 \earyst
  Then we can define
    \begin{equation}
(\tilde{H}(x), \cdots,  \tilde{H}(g^{l(x)-1}(x))) =\Psi_{l(x)}(x,\hat\psi(x),\hat\psi(x)+\Sigma^{l(x)-1}_{i=0}\phi(g^i(x)) ). \label{8.8.7}
\end{equation}
By Proposition  \ref{prop 38}(1), we have
    \begin{equation}
d_{\cH}(\tilde{H}(g^i(x)), H^{(k)}(g^i(x))) < 2\epsilon_{k+1}, \ \forall x\in Z^{k+1}, 0\leq i< l(x). \label{8.8.8}
\end{equation}

For $x \in K^{k+1}$, for each $0 \leq i < l(x)$, we let $\tilde{F}_{i}$ be the lift of $\iota_{\cH}(\tilde{H}(g^{i}(x)))$ that is close to $F^{(k)}_{g^{i}(x)}$.
By $(f2),(f3)$ for $k$, and by Proposition  \ref{prop 38}(2), for any $x\in K^{k+1} = K^{k} \cup Z^{k+1}$, we have
\aryst
\tilde{F}_{l(x)-1} \cdots \tilde{F}_{0}(\hat\psi(x)) &=& \hat\psi(x) + \Sigma^{l(x)-1}_{i=0} \phi(g^i(x)),\\
\tilde{F}_{l(x)-1} \cdots \tilde{F}_{0}([\hat\psi(x)-\bar{r}, \hat\psi(x)+\bar{r}]) &\subset& \hat\psi(x) + \Sigma^{l(x)-1}_{i=0} \phi(g^i(x))+[-\frac{1}{10}\bar{r},\frac{1}{10}\bar{r}].
\earyst

We have the following.
\begin{lemma}\label{lemma 72}
The map $\tilde{H}$ is continuous.
\end{lemma}
\begin{proof}
 It is enough to show that for any $\{x_n\}, \{l_n\}, x', l'$  in Lemma \ref{LEMMA 62} with $j = k + 1$, we have
     \begin{equation}
     \tilde{H}(g^{l_n}(x_n)) \rightarrow  \tilde{H}(g^{l'}(x')), n\rightarrow \infty.
\label{8.8.9}
\end{equation}
We first assume that conclusion (1) in Lemma \ref{LEMMA 62} is true, namely, $x'\in Z^{k+1}$. Then (\ref{8.8.9}) follows immediately from Lemma \ref{LEMMA 59}(3) and the continuity of $\Psi_{l(x')}$.

Now assume that conclusion (2) in Lemma \ref{LEMMA 62} is true, namely, $x'\in  K^k$. It is enough to prove (\ref{8.8.9}) in the following  two cases: (1) $x_n \in K^k$ for all $n$; (2) $x_n \in Z^{k+1}$ for all $n$. 
In the first case, we have $g^{l_n}(x_n) \in W^k$ for all $ n$. By Lemma \ref{LEMMA 62}, we have $g^{l'}(x')\in W^k$. Then (\ref{8.8.9}) follows from (\ref{8.8.6}) and the fact that $F^{(k)}$ is continuous.

Assume that the second case is true, namely, $x_n\in  Z^{k+1}$ for all $n$. Moreover, after passing to a subsequence, we can assume that there exists 
$l_0$ such that $l(x_n) = l_0$ for all $n$. By (\ref{8.8.7}), we have
$$ \tilde{H}(g^{l'}(x_n)) = \mbox{the $l'$-th coordinate of } \Psi_{l_0}(x_n).$$
Then by the continuity of $\Psi_{l_0}$ and the fact that $x_n\rightarrow x', l_n\rightarrow l'$  as $n$ tends to
infinity, we have that
$$ \tilde{H}(g^{l_n}(x_n)) \rightarrow \mbox{the $l'$-th coordinate of } \Psi_{l_0}(x'), n\rightarrow \infty.$$
It is then enough to show that the $l'$-th coordinate of $\Psi_{l_0}(x')$ equals $F^{(k)}_{g^{l'}(x')}$. By Proposition  \ref{prop 38}(3),  it is enough to verify that
\ary
(({F}^{(k)})^{l_0})_{x'}(\hat\psi(x')) &=& \hat\psi(x') + \Sigma^{l_0-1}_{i=0} \phi(g^i(x')),\\
((\tilde{F}^{(k)})^{l_0})_{x'}(\hat\psi(x')+[-\bar{r},\bar{r}]) &\subset& \hat\psi(x') + \Sigma^{l_0-1}_{i=0} \phi(g^i(x'))+[-\frac{1}{10}\bar{r},\frac{1}{10}\bar{r}].
\eary
These follow from $(f2),(f3)$ and Lemma \ref{LEMMA 59}(5).
\end{proof}

By (\ref{8.8.8}) and Proposition \ref{prop-tietze}, we can choose $H^{(k+1)}\in C^0(X, \cH)$   so that $D_{\cal H} (H^{(k+1)}, H^{(k)}) < 2\epsilon_{k+1}$ and satisfies that $H^{(k+1)}(x) = \tilde{H}(x)$ for all $x\in  W^{k+1}$. It is straightforward to verify $(f1)$-$(f3)$
for $k + 1$. This completes the induction.

We let $H' = H^{(d)}$, $f' = \Phi(H')$, and let $\mathcal{R}= \cup_{x\in K}\{x\}\times(\psi(x)-\bar{r}, \psi(x)+\bar{r})$.  By $(f3)$, we can
see that $f'_K(\overline{{\cal R}})\subset \mathcal{R}$. By Lemma \ref{lemma 66}, $f'$  is mode-locked, and hence $H' \in  \mathcal{ML}(\cH)$. This concludes the proof.   
\end{proof}

\begin{proof}[Proof of Theorem \ref{thm:29}]
By Theorem \ref{thm:29'}, it remains to consider the case where  $(\cH, d_{\cH}, \iota_{\cH})$ is given by Example \ref{exam III}.
In this case, $\cH = \R$ and $d_{\cH} = d_{\R}$. Moreover, we may assume that the function $P$ for defining $(\cH, d_{\cH}, \iota_{\cH})$ (see Example \ref{exam II}) has a smaller period, for otherwise $(\cH, d_{\cH}, \iota_{\cH})$ is given by Example \ref{exam II}, and we could already conclude by Theorem \ref{thm:29'}.

Now let $R^{-1} \in (0, 1)$ be the smallest positive period of $P$, where $R \in \Z_{> 0}$ (such $R$ exists since $P$ is non-constant). 
Then there exists a non-constant real analytic function $P \in C^{\omega}(\T)$ with no smaller period such that  $\tilde{P}(R w) = RP(w)$ for every $w \in \T$.  Notice that we have $\| \tilde P' \| = \| P' \| < 1$.
For each $h \in \R$, we define 
\aryst
\tilde{\iota}_{\cH}(h)(w)  = w + \tilde P(w) +  h.
\earyst
Then we have
\aryst
R \iota_{\cH}(h)(w) = \tilde\iota_{\cH}(Rh)(Rw).
\earyst
We define a continuous map $\tilde\Phi : C^0(X, \cH) \to \Diff^{0,1}_g(X \times \T)$ by
\aryst
\tilde\Phi(H)(x, w) = (g(x), \tilde\iota_{\cH}( H(x))(w)).
\earyst
By definition, it is straightforward to deduce the equation
\aryst
\rho(\tilde\Phi(R \cdot H)) = R \rho (\Phi(H)) \in \R/\Z
\earyst
where $R \cdot H$ denotes the function $x \mapsto R H(x)$ in $C^0(X, \cH)$.
By Lemma \ref{claim equivdefmodelock}, we see that $\tilde\Phi(R \cdot H)  \in \cal{ML}$ if and only if $\Phi(H) \in  \cal{ML}$.
By definition, $(\cH, d_{\cH}, \tilde{\iota}_{\cH})$ is given by Example \ref{exam II}. By Theorem \ref{thm:29'}, the set 
$\cal{ML}(H, \tilde\iota_{\cH})$ is dense. Hence $\cal{ML}(H, \iota_{\cH})$ is also dense. This concludes the proof.
\end{proof}

\section{Proof of Lemma \ref{lem mainpropertymovableset}} \label{sec Perturbation lemmata}

Throughout this section, , we always assume that  $(\cH, d_\cH, \iota_\cH)$ is given by either Example \ref{exam I} or  \ref{exam II}.
Moreover, we are working under the fictitious condition that 
\aryst
 C^0(X, \cH) \setminus \overline{ \mathcal{ML}(\cH) } \neq \emptyset.
 \earyst

We fix  some contractable $H \in  C^0(X, \cH) \setminus \overline{ \mathcal{ML}(\cH) }$ such that $L_+(H)>L_-(H)$. We denote $f = \Phi(H)$, and denote by $F$ a ift of $f$.


We have the following result, making use of only  the fact that  $H \in  C^0(X, \cH) \setminus \overline{ \mathcal{ML}(\cH) }$.
\begin{lemma} \label{lem uniformgrowth}
There exist a constant $\kappa = \kappa(H) > 0$ and functions $\Delta, K: \R_+ \to \R_+$ such that the following is true.
For any $\epsilon > 0$,  $n > K(\epsilon)$,  $\tilde H \in \cB_{\cal H}(H, \kappa)$ and $x \in X$, we have 
\aryst
\inf_{y \in \R} [(F_{\epsilon}^n)_x(y)   - (F^n)_{x}(y)] > n \Delta(\epsilon),
\earyst
where $F$ is an arbitrary lift of $\Phi(H)$ (clearly the left hand side above is independent of the choice of the lift).
\end{lemma}

\begin{proof}

Take a constant $\kappa > 0$ such that $\cB_{\cal H}(H, 2 \kappa) \subset C^0(X, \cH) \setminus \overline{ \mathcal{ML}(\cH) }$.

Assume to the contrary that the lemma is false. Then there exist some $\epsilon > 0$, an increasing sequence $\{ n_k \}_{k \geq 1}$, a sequence $\{ H_k \in \cB_{\cal H}(H, \kappa) \}_{k \geq 1}$, a sequence $\{ x_k \in X \}_{k \geq 1}$ and a sequence  $\{ y_k \in [0,1) \}_{k \geq 1}$  such that 
\ary \label{eq growthbecomesslower}
((F^{(k)}_{\epsilon})^{n_{k}})_{x_k} (y_k)   -  ((F^{(k)})^{n_k})_{x_k}(y_k) \leq n_k/k
\eary
where $F^{(k)}$ is a lift of $\Phi(H_k)$.
We may assume without loss of generality that the set $\{ F^{(k)} \}_{k \geq 1}$ is bounded in $C^0(\R, \R)$; and $H_k$ converges to some $\check{H}$. Then, after passing to a subsequence, we may assume that $F^{(k)}$ converges to a lift $\check{F}$ of $\Phi(\check{H})$.

Fix an arbitrary integer $m > 0$. 
By \eqref{eq growthbecomesslower},  for all sufficiently large integer $k > 0$, there exists some $0 \leq  l_k \leq n_k - m$ such that, for $u_k = g^{l_k}(x_k)$ and some $z_k \in [0,1)$, we have
\aryst
((F^{(k)}_{\epsilon})^{m})_{u_k}(z_k)   - ((F^{(k)})^{m})_{u_k}(z_k) \leq 2m/k.
\earyst

By extracting a subsequence, we may assume that $u_k$ converges to some $u \in X$, and $z_k$ converges to some $z \in [0,1]$. Then we have
\ary \label{eq fepsilonfequal0}
(\check{F}_{\epsilon}^m)_{u}(z)   -  (\check{F}^m)_{u}(z) = 0.
\eary
By our choice of $\kappa$, it is clear that $\check{H} \in C^0(X, \cH) \setminus \overline{ \mathcal{ML}(\cH) }$. Thus \eqref{eq fepsilonfequal0} contradicts Corollary \ref{LEMMA 56} if $m$ is sufficiently large.
Consequently, the lemma must be true.
\end{proof}

We have the following corollary.
\begin{cor} \label{cor uniformgrowth}
Let $\Delta, K$ be given as in Lemma \ref{lem uniformgrowth}.
Then there exists a constant $\kappa > 0$ such that the following is true. For any $\epsilon > 0$, $n > K(\epsilon)$,  $x \in X$ and any $(H_{i})_{i=0}^{n-1} \in \cH^{n}$ such that $d_{\cal H}(H_i, H(g^{i}(x))) < \kappa$, we have
\aryst
\inf_{y \in \R} [(F_{n-1})_{\epsilon} \circ \cdots \circ (F_{0})_{\epsilon}(y)   - F_{n-1} \circ \cdots \circ F_{0}(y)] > n \Delta(\epsilon),
\earyst
where $F_i$ is an arbitrary lift of $\iota_{\cal H}(H_i)$ for each $0 \leq i \leq n-1$.
\end{cor}
\begin{proof}
Given $x \in X$ and $n > K(\epsilon)$. We have by Proposition \ref{prop-tietze} that there exists some $\tilde H \in \cB_{\cal H}(H, \kappa)$ and $\tilde H(g^{i}(x)) = H_i$ for every $0 \leq i \leq n-1$. Then we can immediately deduce the corollary by Lemma \ref{lem uniformgrowth}.
\end{proof}

In the following lemma, we will use the fact that  $(\cH, d_\cH, \iota_\cH)$ is given by either Example \ref{exam I} or  \ref{exam II}.
\begin{lemma} \label{lem getparabolicelement}
For every $\epsilon > 0$, there exist an integer $M = M(H, \epsilon) > 0$ and a function $\delta_{\epsilon, M}: \R_+ \to \R_+$ such that for any $x \in X$ and $y \in \R$, there exists $( p^{(j)}_{i} )_{i = 0}^{M - 1} \in {\cal H}^{M}$, $j \in \{ 0,1 \}$ such that $d_{\cal H}(p^{(j)}_i, H(g^{i}(x))) < \epsilon$ for $j \in \{0,1\}$ and $0 \leq i \leq M -1$,
and the following is true. Denote by $P^{(0)}_i$ and $P^{(1)}_i$ lifts of $\iota_{\cal H}(p^{(0)}_i)$ and $\iota_{\cal H}(p^{(1)}_i)$ respectively, which are close to each other. Then we have
\enmt
\item $P^{(1)}_{M -1} \circ \cdots \circ P^{(1)}_{0}(y) = P^{(0)}_{M -1} \circ \cdots \circ P^{(0)}_{0}(y)$,
\item  $P^{(1)}_{M -1} \circ \cdots \circ P^{(1)}_{0}(z) > P^{(0)}_{M -1} \circ \cdots \circ P^{(0)}_{0}(z) +  \delta_{\epsilon, M}(\sigma)$ for every $z \notin y + (- \sigma, \sigma) + \Z$ and for every $\sigma > 0$.
\eenmt
\end{lemma}

\begin{proof}

This lemma is obvious if  $(\cH, d_\cH, \iota_\cH)$ is in Example \ref{exam I}, as we can use the projective action given by the parabolic elements in $SL(2, \R)$ to make perturbations.

Now we assume that $(\cH, d_\cH, \iota_\cH)$ is in Example \ref{exam II}.

Fix an arbitrary $x \in X$. We define a real analytic function $c_x :    \T \times \R^2 \to \R$ by
\aryst
c_x(y \mod 1, s, t) = F_{g^2(x)} \circ(F_{s+t})_{g(x)} \circ (F_t)_x(y) - F_{g^2(x)} \circ F_{g(x)} \circ F_x(y).
\earyst
By straightforward computation, we see that
\ary \label{eq p -1}
 (f_{- c_x(w, s, t)})_{g^2(x)}  \circ (f_{s + t})_{g(x)} \circ (f_t)_x(w)   =  f_{g^2(x)}  \circ f_{g(x)} \circ f_x(w).
\eary
By continuity, we have
\ary \label{eq cx 0}
\lim_{s \to 0} \sup_{x \in X} \sup_{w \in \T} |c_x(w, s, 0)| = 0.
\eary
Moreover, given any $w_0 \in \R$, we have 
\aryst
&& \partial_t \{ (f_{- c_x(w_0, s, t)})_{g^2(x)}  \circ (f_{s + t})_{g(x)} \circ (f_t)_x(w) \}|_{t = 0}  \\
&=& - \partial_t c_x(w_0, s, t)|_{t = 0} + \partial_t \{ f_{g^2(x)} \circ (f_{s + t})_{g(x)} \circ (f_t)_x(w) \}|_{t = 0} \\
&=&  - \partial_t c_x(w_0, s, t)|_{t = 0} + \partial_t \{ c_x(w, s, t) \}|_{t = 0}.
\earyst

Fix some $s \in (0, \epsilon / 2)$ sufficiently close to $0$ such that $\sup_{x \in X} \sup_{w \in \T} |c_x(w, s, 0)| < \epsilon / 2$, and  
the function $ \partial_t \{ c_x(w, s, t) \}|_{t = 0}$, as a non-constant real analytic function of $w$, reaches its minimum on a finite set $C_x \subset \T$.
Then
\aryst
 \partial_t \{ (f_{- c_x(w, s, t)})_{g^2(x)}  \circ (f_{s + t})_{g(x)} \circ (f_t)_x(w) \}|_{t = 0} \geq 0,
\earyst
and the equality holds if and only if $w \in C_x$. By compactness, there exists an integer $L > 0$ such that $|C_x| \leq L$ for every $x \in X$.

Define $M = 5L$. We will inductively construct $p^{(j)}_i$ for $j \in \{0,1 \}$ and $i \in \{ 0, \cdots, M-1 \}$ so that the conclusion of the lemma is satisfied.

Denote $C_x = \{ w_0, \cdots, w_{\ell - 1} \}$ where $\ell \leq L$.
We set $w_{0, i} = w_i$ for $0 \leq i \leq \ell-1$.

Given an integer $0 \leq m \leq \ell - 1$. Assume that  we have already constructed  $p_i$ for $i \in \{ 0, \cdots, 5m - 1 \}$.
We define 
\ary \label{eq p 0} 
&& (p_{5m}, p_{5m + 1}, p_{5m + 2})  \\
&=&  ( H( g^{5m}(x)), H( g^{5m + 1}(x) ), H( g^{5m+2}(x) ) ) + (0, s, -c_{g^{5m}(x)}(  (f^{5m})_x (w_0), s, 0)). \nonumber
\eary
By our choice of $s$, we have $(p_{5m}, p_{5m + 1}, p_{5m + 2}) \in \cB_{\cH}( H(g^{5m}(x)), \epsilon) \times  \cB_{\cH}( H(g^{5m + 1}(x)), \epsilon) \times  \cB_{\cH}( H(g^{5m + 2}(x)), \epsilon)$.
We define 
\aryst
z_{m, i} = \iota_{\cH}( ( p_{5m + k} )_{k = 0}^{2} )(w_{m, i}), \quad \forall 0 \leq i \leq \ell-1.
\earyst
By Lemma \ref{lem moveonefixone}, 
there exists $(p_{5m + 3}, p_{5m + 4}) \in \cB_{\cH}( H(g^{5m + 3}(x)), \epsilon) \times  \cB_{\cH}( H(g^{5m + 4}(x)), \epsilon)$ such that 
\ary
\iota_{\cH}((p_{5m + 3}, p_{5m + 4}))(z_{m, 0}) &=& \iota_{\cH}( (H(g^{5m + 3}(x)),  H(g^{5m + 4}(x)) ) )(z_{m, 0}), \label{eq p 1}  \\
\iota_{\cH}((p_{5m + 3}, p_{5m + 4}))(z_{m, m + 1}) &\not\in& C_{g^{5m + 5}(x)} \ \mbox{ if $m < \ell - 1$.} \label{eq p 2}
\eary
We define 
\aryst
w_{m+1, i} = \iota_{\cH}( (p_{5m + 3}, p_{5 m + 4}) )(z_{m, i}), \quad \forall 0 \leq i \leq \ell - 1.
\earyst

We set $p^{(0)}_{k} = p_{k}$ for every $0 \leq k \leq 5\ell -1$.
Let $t > 0$ be a small constant to be determined. For each $0 \leq m \leq \ell -1$, we define 
\aryst
& (p^{(1)}_{5m}, p^{(1)}_{5m + 1}, p^{(1)}_{5m + 2}) =  ( H( g^{5m}(x)), H( g^{5m + 1}(x) ), H( g^{5m+2}(x) ) ) + (t, s + t, -c_{g^{5m}(x)}(  (f^{5m})_x (w_0), s, t)),  \\
& (p^{(1)}_{5m + 3},   p^{(1)}_{5m + 4} ) = (p_{5m + 3},  p_{5m + 4}). 
\earyst
We set $p^{(0)}_{k} = p^{(1)}_{k} = H(g^{k}(x))$ for every $5 \ell \leq k \leq M - 1$.
By letting $t$ be sufficently close to $0$, we have $d_{\cal H}(p^{(j)}_k, H(g^{k}(x))) < \epsilon$ for $j \in \{0,1\}$ and $0 \leq k \leq M - 1$. 

Take an arbitrary $y \in \R$ such that $y \mod 1 = w_0$.
It is then straightforward to verify that Item (1) by \eqref{eq p 0} and \eqref{eq p -1}.
To verify Item (2), we take an arbitrary $u \in \T$ and denote $u_k = \iota_{\cH}( (  p_i )_{i = 0}^{k -  1} )(u)$ for each $0 \leq k \leq M-1$. 
Now we view $p^{(1)}_k$ for each $0 \leq k \leq M-1$ as a function of $t$.  Then we have
\aryst
&& \partial_t \{ \iota_{\cH}(  ( p^{(1)}_k )_{k = 0}^{M - 1} )(u) \}|_{t = 0} \\
&=& \sum_{ m = 0 }^{ \ell - 1} D\iota_{\cH}(  ( p_k )_{k = 5m + 3 }^{M-1} )( u_{5m + 3} )  \\
&& \cdot  \partial_t \{ (f_{- c_{ g^{5m}(x) }(   (f^{5m})_x (w_0) , s, t)})_{g^{5m + 2}(x)}  \circ (f_{s + t})_{g^{5m + 1}(x)} \circ (f_t)_{ g^{5m}(x) }(  u_{5m} ) \}|_{t = 0}.
\earyst
By construction, we see that for any $u \neq w_0$, there exists some $0 \leq m \leq \ell - 1$ such that $u_{5m} \notin C_{g^{5m}(x)}$. Consequently, we have 
\aryst
 \partial_t \{ \iota_{\cH}(  ( p^{(1)}_k )_{k = 0}^{M - 1} )(u) \}|_{t = 0} \geq 0
\earyst
with equaity if and only if $u = w_0$. Then it is straightforward to deduce Item (2) by compactness.
\end{proof}

Now we state the main observation for the proof of Lemma \ref{lem mainpropertymovableset}.
\begin{lemma}\label{lemma 67} For any $\epsilon > 0$, there exists a constant $N_2 = N_2(H,\epsilon) > 0$ such that the following is
true. For any $(x, y) \in X \times \mathbb{R}$, there exist integers $M_-, M_+$ satisfying $-N_2 \leq M_- \leq 0\leq M_+ \leq N_2$ such that for any $z_+,z_- \in \mathbb{R}$ with $|z_{+}- (F^{M_+})_x(y)|,|z_{-} - (F^{M_{-}})_x(y)|<2$,
there exist $(p_{M_{-}}, \cdots, p_{M_+-1} ) \in \cH^{M_+-M_-}$ and $y'\in (y,y+1)$ 
such that
\begin{enumerate}
    \item $d_{\cH}(p_i,H(g^i(x)))<\epsilon$ for any $M_- \leq i\leq M_{+}-1$;
    \item Denote by $P_i$ the unique lift of $\iota_{\cH}(p_i)$ close to $F_{g^{i}(x)}$.
    Then either we have
    \[
    \begin{array}{ccc}
      & P_{0} \circ  \cdots \circ P_{M_+ - 1}[y,y') \subset (z_+ -\epsilon, z_+ +\epsilon)\\
     \mbox{ and } & (P_{-1} \circ \cdots \circ P_{M_-} )^{-1}[y',y+1) \subset (z_- -\epsilon,z_- +\epsilon), 
    \end{array}
    \]
\end{enumerate}
or we have
\[
    \begin{array}{ccc}
      &  P_{0} \circ  \cdots \circ P_{M_+ - 1}[y',y+1) \subset (z_+ -\epsilon, z_+ +\epsilon)\\
      \mbox{ and }  & (P_{-1} \circ \cdots \circ P_{M_-} )^{-1}[y,y') \subset (z_- -\epsilon,z_- +\epsilon). 
    \end{array}
\]
 \end{lemma}
 \begin{proof}

  Without loss of generality, we assume that $\epsilon\in  (0, \min(\kappa(H), 1)/2)$ where
 $\kappa(H)$ is given by  Lemma \ref{lem uniformgrowth}. 

  We let $\epsilon_0 = \epsilon_0(H, \epsilon/2)$ be given by Lemma \ref{lem cancellation}.
   Let 
\ary
     m_0 &=& \lceil 3 \kappa_1(f,  \epsilon_0)^{-1}N_1(f, \epsilon_0) \rceil +1,\label{equation (8.5.2)}  \\ 
 \epsilon_2 &=& (\Vert f \Vert_{C^{0,1}} + 1)^{-m_0}\epsilon  \label{equation (8.5.3)}
\eary
where $\kappa_1$ and $N_1$ are given by Corollary \ref{LEMMA 56}.

Let $M = M(H, \epsilon)$ be given by Lemma \ref{lem getparabolicelement}.
Given $x \in X$ and $y \in \R$, we apply Lemma \ref{lem getparabolicelement} to obtain
$(p^{(j)}_i)_{i=0}^{M-1} \in {\cal H}^{M}$ for $j \in \{ 0,1 \}$ and $0 \leq i \leq M-1$. We denote
$(p^{(j)}_i)_{i=0}^{M-1}$ as $(p^{(j)}_{x, y, i})_{i=0}^{M-1}$ to indicate the dependence on $x$ and $y$.

By Lemma \ref{lem getparabolicelement}, we define
\ary
\epsilon_1  &=& 
\delta_{\epsilon, M}( \epsilon_2) > 0, \\
\epsilon'_1 &=&M^{-1}( \Vert f\Vert_{C^{0,1}} + 1)^{-M} \epsilon_1.
\eary
We have for every $x \in X$, $y \in \R$ and $y' \not\in \mathbb{Z} + y + (-\epsilon_2,\epsilon_2)$ that
\ary
 P^{(1)}_{x, y, M - 1} \circ \cdots \circ P^{(1)}_{x, y, 0}(y') - P^{(0)}_{x, y, M - 1} \circ \cdots \circ P^{(0)}_{x, y, 0}(y') >  \epsilon_1.   \label{equation (8.5.6)} 
\eary
We let
\ary
   m_1 =   3 \lceil \Delta(\epsilon'_1)^{-1} \rceil K(\epsilon'_1) \label{equation (8.5.7)}
\eary
where functions $\Delta, K$ are given by Lemma \ref{lem uniformgrowth}.

We fix some $x \in X$ and $y \in \R$ from now on.
We define
\ary \label{equation (8.5.8)} \\
    \hat{p}_{kM + i} = p^{(1)}_{g^{kM}(x), (F^{kM})_x(y), i}, \  \check{p}_{kM + i} = p^{(0)}_{g^{kM}(x), (F^{kM})_x(y), i}, \quad \forall k \in \Z, 0 \leq i \leq M-1.   \nonumber
\eary

By Lemma \ref{lem getparabolicelement}, we have $d_{\cH}( \hat{p}_i, H(g^i(x)) ), d_{\cH}( \check{p}_i, H(g^i(x)) ) <\epsilon$ for all $i\in\mathbb{Z}$.
We denote by $\hat{P}_i$, resp. $\check{P}_i$,  the unique lift of $\hat{p}_i$, resp. $\check{p}_i$, close to $F_{g^{i}(x)}$.
 By (\ref{equation (8.5.8)}) and Lemma \ref{lem getparabolicelement}(1), 
 it is direct to verify that for any $i \geq 0, \hat{P}_{i M -1} \cdots \hat{P}_0(y) = \check{P}_{i M -1} \cdots \check{P}_0(y) = (F^{i M})_{x}(y)$; 
 and for any $i < 0, \hat{P}^{-1}_{i M} \cdots \hat{P}^{-1}_{-1}(y) = \check{P}^{-1}_{i M}, \cdots \check{P}^{-1}_{-1}(y) = (F^{i M})_x(y)$.

Define $$N_2 = m_1 M + m_0.$$
We have the following claim.
\begin{claim} \label{claim soitcasoitca}
For any $z\in(y,y+1)$, there exists $0\leq i\leq m_1$ such that $$\hat{P}_{iM-1}\cdots \hat{P}_0(z)\in (F^{iM})_x(y)+(0,\epsilon_2) \cup (1-\epsilon_2,1).$$
Similarly, for any $z\in (y,y+1)$, there exists $0\leq i\leq m_1$ such that $$\hat{P}^{-1}_{-iM}\cdots \hat{P}^{-1}_{-1}(z) \in (F^{-iM})_x(y)+  (0,\epsilon_2)\cup (1-\epsilon_2,1).  $$
\end{claim}
\begin{proof} 
 We will only detail the first statement, since the second one follows from a similar argument. If the first statement is false, let $z_i = \hat{P}_{iM -1}\cdots \hat{P}_0(z)$ for all $0 \leq i \leq m_1$, then   we would have
\aryst
z_i \not\in\mathbb{Z}+ (F^{i M})_x(y)+ (-\epsilon_2,\epsilon_2), \ \forall 0\leq i\leq m_1.
\earyst
Then by  (\ref{equation (8.5.6)}), we have
\ary
  \hat{P}_{(i+1)M - 1} \circ \cdots \circ  \hat{P}_{iM}(z_i) &>&   \check{P}_{(i+1)M - 1} \circ \cdots \circ  \check{P}_{iM}(z_i) + \epsilon_1 \nonumber \\
  & \geq &  (\check{P}_{(i+1)M - 1})_{\epsilon'_1} \circ \cdots \circ  (\check{P}_{iM})_{\epsilon'_1}(z_i). \label{equation (8.5.9)}
\eary
By Corollary \ref{cor uniformgrowth},  (\ref{equation (8.5.7)})  and (\ref{equation (8.5.9)}), we would have
\aryst
&&     \hat{P}_{m_1 M -1}\cdots \hat{P}_0(y+1)\\ 
&\geq& \hat{P}_{m_1 M -1} \cdots \hat{P}_0(z) \geq  (\check{P}_{m_1M - 1})_{\epsilon'_1} \circ \cdots \circ  (\check{P}_{0})_{\epsilon'_1}(y) >  (F^{m_1 M})_x(y)+2.
\earyst
This is a contradiction.
\end{proof}
 
Let
\begin{align*}
    U_- &=\{z\in (y,y+1)\mid \mbox{there exists} \ 0\leq i\leq m_1 \mbox{ such that } \hat{P}_{iM-1}\cdots \hat{P}_0(z)\in (F^{iM})_x(y)+(0,\epsilon_2)\}, \\
    U_+ &=\{z\in (y,y+1)\mid \mbox{there exists} \  0\leq i\leq m_1 \mbox{ such that } \hat{P}_{iM-1}\cdots \hat{P}_0(z)\in (F^{iM})_x(y)+(1-\epsilon_2,1)\}.
\end{align*}
Then by our claim, we have
$$U_- \cup U_+ = (y,y+1).$$
It is direct to see that $U_-, U_+$ are both non-empty connected open sets. Since $(y, y + 1)$ is connected, we conclude that $U_-  \cap U_+\neq\emptyset$. We take $y'$ to be an arbitrary element in $U_- \cap   U_+$.

Again by Claim \ref{claim soitcasoitca}, there exists $0 \leq n'_- \leq m_1$ such that for $N'_- = M n'_-$ we have  $$\hat{P}^{-1}_{-N'_{-}} \cdots \hat{P}^{-1}_{-1}(y')\in (F^{-N'_{-}})_x(y)+(0,\epsilon_2)\cup(1-\epsilon_2,1).$$ Without loss of generality, we may assume that
$$\hat{P}^{-1}_{-N'_{-}} \cdots \hat{P}^{-1}_{-1}(y')\in (F^{-N'_{-}})_x(y)+(0,\epsilon_2),$$
 as the other case can be dealt with by a similar argument. 
Then we have
  $$\hat{P}^{-1}_{-N'_{-}} \cdots \hat{P}^{-1}_{-1}[y,y')\subset (F^{-N'_{-}})_x(y)+[0,\epsilon_2).$$

By $y' \in U_+$, we see that  there exists $n'_+ \in \{0, \cdots,  m_1 \}$ such that for $N'_+ = M n'_+$ we have
\aryst
\hat{P}_{N'_{+}-1} \cdots \hat{P}_0(y') \in (F^{N'_+})_x(y)+(1-\epsilon_2,1).
\earyst
 Hence 
    \begin{equation}
    \hat{P}_{N'_{+}-1} \cdots \hat{P}_0[y',y+1)\subset (F^{N'_+})_x(y)+(1-\epsilon_2,1).\label{equation (8.5.10)}
    \end{equation}
We  set
$$M_+=N'_+ +m_0 \ \mbox{ and } \ p_i = \hat{p}_i,  \ P_i=\hat{P}_i, \ \forall 0 \leq i \leq N'_+ -1.$$
By  Lemma \ref{lem cancellation}, (\ref{equation (8.5.2)}),  and the fact that
$$(F^{m_0})_{g^{N'_+}(x)}\hat{P}_{N'_{+}-1} \cdots \hat{P}_0(y+1) = (F^{M_+})_x(y+1)\in z_+ +(-3,3),$$
there exists $(p_{N'_+}, \cdots, p_{M_+ -1} )\in \cH^{m_0}$ such that  
\begin{enumerate}
    \item $d_{\cH}(p_{N'_+ +i}, H(g^{N'_+ +i}(x)) )<\epsilon$ for any $0\leq i\leq m_0-1$,
    \item $P_{M_+ -1}\cdots P_{N'_+} ( (F^{N'_+})_x (y)+1)=z_+$ 
\end{enumerate}
where $P_i$ is a  lift of $\iota_{\cal H}(p_i)$ close to $F_{g^{i}(x)}$ for each $N'_+ \leq i \leq M_+ - 1$.
Then by (\ref{equation (8.5.10)}) and  (\ref{equation (8.5.3)}),  we have
\aryst
    P_{M_+-1}\cdots P_0[y',y+1) &\subset& P_{M_+-1}\cdots P_{N'_+}( (F^{N'_+})_x(y)+(1-\epsilon_2,1])\\
    &\subset& (z_+-\epsilon, z_++\epsilon).
\earyst

 By a similar method, we may set
 \aryst
 M_- =  - N'_- - m_0,
 \earyst
 and define $(p_{M_-}, \cdots p_{-1})$ and a lift $P_i$ of $\iota_{\cal H}(p_i)$ for each $M_- \leq i \leq - 1$. It is then direct to verify (1) and (2).
\end{proof}

We also need the following two lemmata. 
In the following, we denote by $y_i$ a point in $\R$, and denote by $\bar{y}_i = y_i \mod \Z \in \T$.
\begin{lemma}\label{lemma 68}  
 For any $\epsilon, \eta > 0$, there exist $\epsilon_3 = \epsilon_3(H,\epsilon) > 0$, $r_1 = r_1(H,\epsilon, \eta) > 0$,
$N_3 = N_3(H , \epsilon) > 0$ such that for any integer $N \geq N_3$ the following is true:

(\textit{Forward contraction}) For any $x\in X$, there exist $y_1\in\mathbb{R}$ and $(q_0,\cdots, q_{N-1})\in \cH^N$ such that
\begin{enumerate}
\item[(1)] $d_{\cH}(q_{i}, H(g^{i}(x)))<\epsilon$ for every $0\leq i\leq N-1$, 
\item[(2)] 
 $D(\iota_{\cH}((q_j)_{j=0}^{i-1}))(\bar{y}_1)<e^{-2\epsilon_3i}$ for every $N_3 \leq i\leq N$,
\item[(3)]  $\frac{D\iota_{\cH}((q_j)_{j=0}^{i -1})(\bar{y}') }{D\iota_{\cH}((q_j)_{j=0}^{i -1})(\bar{y}_1) } \in (e^{- \min(\epsilon_3, \eta) i}, e^{ \min(\epsilon_3, \eta) i})$ for every $0 \leq i \leq N$ and every $\bar{y}'\in  (\bar{y}_1- r_1, \bar{y}_1 + r_1)$, 
\item[(4)] $\iota_{\cH}((q_j)_{j=0}^{N -1})(\bar{y}_1)= (f^{N})_x(\bar{y}_1)$;
\end{enumerate}

(\textit{Backward contraction}) For any $x \in X$, there exist $y_2 \in \R$ and $(q_{-N}, \cdots, q_{-1}) \in  \cH^N$ such that 
\begin{enumerate}
\item[(1')] $d_{\cH}(q_{-i}, H(g^{-i}(x)))<\epsilon$ for every $1\leq i\leq N$,
\item[(2')] $D(\iota_{\cH}((q_j)_{j=-i}^{-1})^{-1})(\bar{y}_2) <e^{-2\epsilon_3i}$ for every $N_3 \leq i\leq N$,
\item[(3')] $\frac{D( \iota_{\cH}((q_j)_{j= -i}^{-1})^{-1} )(\bar{y}') }{D( \iota_{\cH}((q_j)_{j= -i}^{-1})^{-1}) (\bar{y}_2) }  \in (e^{-\min(\epsilon_3, \eta) i}, e^{\min(\epsilon_3, \eta) i})$ for every $0 \leq i \leq N$ and every $\bar{y}'\in  (\bar{y}_2- r_1, \bar{y}_2 + r_1)$, 
\item[(4')] $\iota_{\cH}((q_j)_{j=0}^{N -1})(\bar{y}_2)= (f^{N})_x(\bar{y}_2)$.
\end{enumerate}
\end{lemma} 
\begin{proof}

We will detail the proof of the case (Forward contraction). The other case follows from a similar argument.

We fix a small constant  $\sigma_0 > 0$ such that for every $x \in X$ and $w \in \T$ we have
\aryst
E(\sigma_0, w, x, H) \in \prod_{i=0}^{k-1} \cB_{\cal H}(H(g^{i}(x)), \epsilon),
\earyst
where the map $E$ is given by  Lemma \ref{lem arnoldfamilyismovable}.
We let $r_0, \epsilon_2$ be given by Lemma \ref{lem arnoldfamilyismovable}(3) for $\sigma_0$.

Fix an arbitrary $x \in X$. 
For any integer $n\geq 1$, we define
$$ A_{n,\epsilon_2} = \{ w  \in \mathbb{T}\mid D(f^n)_x(w)> 1000 k \epsilon_2^{-1} e^{n\epsilon_2/(100k)}\}.$$
By identity (\ref{(8.3.1)}) and Markov’s inequality, we have
$$| \bigcup_{n\geq 1}A_{n,\epsilon_2}| \leq \Sigma_{n\geq 1} |A_{n,\epsilon_2}|<\Sigma_{n\geq 1} \frac{\epsilon_2}{1000 k}e^{-n\epsilon_2/(100k)}<1.$$
 We fix an arbitrary $\bar{y}_1\in\mathbb{T}\setminus(\cup_{n\geq 1} A_{n,\epsilon_2})$. 
 
 We  let $N_3 > 0$ be a large integer to be determined depending only on $H$ and $\epsilon$,
 and let
 \ary
 \epsilon_3 = \epsilon_2/(100 k). \label{eq epsilon3}
  \eary
By the choice of $\bar{y}_1$, and by letting $N_3$ be sufficiently large, we have
\begin{equation}
D(f^n)_x(\bar{y}_1)\leq 1000 k \epsilon_2^{-1}e^{n\epsilon_2/(100k)} < e^{n\epsilon_2/(50k)} = e^{2n\epsilon_3},  \ \forall n\geq  N_3.
 \label{formula 8.6.4}
\end{equation}
For any $N\geq  N_3$, we define
\begin{equation}
(q_{ik+j} )_{j=0}^{k-1} = E(\sigma_0, (f^{ik})_x(\bar{y}_1), g^{ik}(x), H), \  \forall 0 \leq i \leq \lfloor (N - 1)/k \rfloor - 1
 \label{formula 8.6.5}
\end{equation}
and define 
\aryst
q_j = H(g^{j}(x)), \ \forall k \lfloor (N - 1)/k \rfloor  \leq j < N.
\earyst
Then item (1) follows from our choice of $\sigma_0$. It is direct to verify (4) by Lemma \ref{lem arnoldfamilyismovable}(2).

We let $r'>0$ be a small constant depending only on $H, {\cal H}, k$ and $\epsilon$ such that 
\begin{equation}
\frac{D(f^{k})_{x'}(\bar y'')}{D(f^{k})_{x'}(\bar y')} < e^{\min( \epsilon_3, \eta) / 2}, \ \forall x' \in X, |\bar{y}'- \bar{y}''| < 2r'
 \label{formula 8.6.3'}
\end{equation} 
and 
\ary
 \frac{DE(\sigma_0, w, x', H)(\bar y'')}{DE(\sigma_0, w, x', H)(\bar y')} < e^{\min( \epsilon_3, \eta) / 2}, \ \forall x' \in X,  |\bar y'- \bar y''| < 2r'.
 \label{formula 8.6.3}
\eary

By \eqref{eq epsilon3}, \eqref{formula 8.6.4}, \eqref{formula 8.6.5}  and Lemma \ref{lem arnoldfamilyismovable}(3),   we have
\begin{equation}
D(\iota_{\cH}((q_i)_{i=0}^{n-1}))(\bar{y}_1)< e^{- n\epsilon_2/2k} = e^{ - 50 n \epsilon_3}, \ \forall N_3 \leq n \leq N.
 \label{formula 8.6.6}
\end{equation}
This proves (2). We choose $r_1 = r_1(H, \epsilon, \eta) \in ( 0, r')$ to be sufficiently small, so that we have
$$\iota_{\cH}((q_i)_{i=0}^{n-1})(\bar{y}_1-r_1, \bar{y}_1 +r_1)\subset (f^n)_x(\bar{y}_1) + (-r',r'), \ \forall 0\leq n\leq N_3.$$ 
 By (\ref{formula 8.6.3}), (\ref{formula 8.6.6}) and a simple induction, we obtain (3).
\end{proof}

\begin{lemma}\label{lemma 69}
  For any $\eta >0$, there exists $r_2 =r_2(H,\eta)>0$ such that for any integer $N\geq 1$ the
following is true:

(Forward expansion) For any $x\in X$, there exists $\bar{y}_3 \in \mathbb{T}$ such that $D(f^N)_x(\bar{y}') >
e^{-\eta N}$ for any $\bar{y}' \in (\bar{y}_3 - r_2\Vert f\Vert^{-N}_{C^{0,1}}, \bar{y}_3+r_2\Vert f\Vert^{-N}_{C^{0,1}})$,

(Backward expansion) For any  $x\in X$, there exists $\bar{y}_4 \in \mathbb{T}$ such that $D(f^{-N})_x(\bar{y}') >
e^{-\eta N}$ for any $\bar{y}'\in (\bar{y}_4 - r_2\Vert f\Vert^{-N}_{C^{0,1}}, \bar{y}_4+r_2\Vert f\Vert^{-N}_{C^{0,1}}).$
\end{lemma}

\begin{proof}
By (\ref{(8.3.1)}), we can choose $\bar{y}_3 \in\mathbb{R}$ so that $D(f^N)_x(\bar{y}_3) = 1$. Then by letting $r_2$ be sufficiently small, and by continuity, we can verify (Forward expansion). The proof of (Backward expansion) is similar.
\end{proof}

\begin{proof}[Proof of Lemma \ref{lem mainpropertymovableset}]
As before we denote $f = \Phi(H)$.
We set
\aryst
D= \Vert f\Vert_{C^{0,1}} + 3 \ \mbox{(so that $\log D \geq 1$)}, \quad  \eta =  (\frac{L_+(f)-L_-(f)}{100 \log D })^2 > 0.
\earyst
We denote $n_1 =  N'_0(H, \eta) > 0$ where $N'_0$ is given by in Lemma \ref{LEMMA limitLE}. Then for any $n\geq   n_1$, we have
\ary \label{8.6.7}
e^{(L_-(f)-\eta)n} <  D(f^n)_x( w ) < e^{(L_+(f)+\eta)n}, \ \forall (x, w)\in X\times\mathbb{T}. 
\eary
By continuity and subadditivity, it is direct to see that there exists $\epsilon'=\epsilon'(H, \eta)>0$ such that  for any
$n\geq n_1$, for any $(h_0,\cdots, h_{n-1})\in \cH^n$ satisfying $d_{\cH}(h_i,H(g^i(x)))<\epsilon'$ for all $0\leq i\leq n-1$, we have
\aryst
 e^{(L_-(f)-2\eta)n} <  D(\iota_{\cH}((h_i)_{i=0}^{n-1}))(w) < e^{(L_+(f)+2\eta)n}, \ \forall (x, w) \in X \times \mathbb{T}. 
 \earyst
Without loss of generality, we can assume that $\epsilon\in (0, \epsilon')$.
We let $\epsilon_3 =\epsilon_3(H,\epsilon), r_1 =r_1(H,\epsilon, \eta), N_3 =N_3(H,\epsilon)$ be given by Lemma \ref{lemma 68}.  
We let $r_2 = r_2(H, \eta) > 0$ be given by Lemma \ref{lemma 69}. Denote
\aryst
r_0 = r_0(H, \epsilon) =  \min( r_1, r_2,\epsilon)/2.
\earyst
Let $N_2 = N_2(H, r_0)$ be given by Lemma \ref{lemma 67}. It is clear that ultimately $N_2$ depends only on $H$ and $\epsilon$.

We let $\epsilon_0 = \epsilon_0(H, \epsilon/2)$ be given by Lemma \ref{lem cancellation}.
Define
\ary 
m_2 = \lceil 2\kappa_1(\Phi(H), \epsilon_0)^{-1}N_1(\Phi(H), \epsilon_0)\rceil+1   \label{eq m2} 
\eary
where $\kappa_1, N_1$ are given Lemma \ref{lem cancellation}.

We let $N_4 > 0$ be a large integer to be determined depending only on $H$ and  $\epsilon$.
Take an arbitrary integer $N>N_4$, we set
$$x_0=g^{\lceil \frac{N}{2}\rceil}(x).$$
We apply Lemma \ref{lemma 67} for $(r_0, x_0, 0)$ in place of $(\epsilon, x, y)$ to obtain $(M_1, M_2)$ in place of $(M_+,M_-)$. We have $-N_2 \leq M_2 \leq 0\leq M_1 \leq N_2$.

By letting $N_4$ be sufficiently large, we have $N -\lceil \frac{N}{2}\rceil - M_1>N_3$.   Then by Lemma \ref{lemma 68} (Forward  
 contraction) for $(g^{M_1}(x_0),N-\lceil\frac{N}{2} \rceil-M_1)$ in place of $(x,N)$, we obtain  $\bar{y}_1\in \T$ 
 and $\tilde{p}_0,\cdots,\tilde{p}_{N-\lceil\frac{N}{2} \rceil-M_1-1} \in \cH$, such that

 \begin{itemize}
     \item[g1.] $d_{\cH}(\tilde{p}_i,H(g^{M_1+i}(x_0)) ) <\epsilon$ for all $0\leq i\leq N-\lceil\frac{N}{2}\rceil-M_1-1$,
     \item[g2.] $D(\iota_{\cH}((\tilde{p}_{j})_{j = 0}^{i-1})(\bar{y}_1) < e^{-2\epsilon_3 i}$ for all $N_3\leq i\leq N-\lceil\frac{N}{2}\rceil-M_1-1$,
     \item[g3.] $\frac{D \iota_{\cH}((\tilde{p}_j)_{j=0}^{i-1})(\bar{y}') }{D \iota_{\cH}((\tilde{p}_j)_{j=0}^{i-1})(\bar{y}_1) } \in (e^{-\min(\eta,\epsilon_3) i}, e^{\min(\eta,\epsilon_3) i})$  for any $0 \leq i\leq N-\lceil \frac{N}{2}\rceil-M_1-1$ and any $\bar{y}'\in (\bar{y}_1-r_1, \bar{y} + r_1)$,
     \item[g4.] $\iota_{\cH}((\tilde{p}_j)_{j=0}^{N-\lceil \frac{N}{2}\rceil-M_1-1})(\bar{y}_1)  = ( f^{N-\lceil \frac{N}{2}\rceil-M_1})_{g^{M_1}(x_0)}(\bar{y}_1)$.
 \end{itemize}

 Without loss of generality we may assume that the lifts of $\bar{y}_1$, $\bar{y}'_1$, denoted by $y_1, y'_1 \in \R$ respectively, satisfy that
\aryst
 \vert (F^{M_1})_{x_0}(0)-y_1\vert, \vert (F^{M_2})_{x_0}(0)-y'_1\vert <2.
 \earyst
We denote by $\tilde{P}_i$ the unique lift of $\iota_{\cH}(\tilde{p}_i)$ close to $F_{g^{M_1 + i}(x_0)}$ for every $0 \leq i \leq N -\lceil\frac{N}{2}\rceil- M_1- 1$. Then by $r_0 < r_1$ and (g3), for any $1\leq i\leq N -\lceil\frac{N}{2}\rceil- M_1$ we have
\begin{equation}
\frac{D(\tilde{P}_{i-1}\cdots\tilde{P}_0)(y')}{D(\tilde{P}_{i-1}\cdots\tilde{P}_0)(y_1)}
\in (e^{-i\min(\eta,\epsilon_3)},e^{i\min(\eta,\epsilon_3)}), \ \forall y'\in(y_1-r_0,y_1+r_0).
\label{8.6.9}
\end{equation}
By Lemma \ref{lemma 67} for $(r_0, y_1, y'_1)$ in place of $(\epsilon, z_+, z_-)$, 
we obtain $y'_0\in (0,1)$ and $(\hat{p}_{M_2},\cdots, \hat{p}_{M_1-1})\in  {\cal H}^{M_1-M_2}$ such that $d_{\cH}(\hat{p}_i,H( g^i(x_0)) )<r_0\leq \epsilon$ for all $M_2\leq i\leq M_1-1$, and, denote by $\hat{P}_i$ the lift of $\iota_{\cH}(\hat{p}_i)$ close to $F_{g^{i}(x_0)}$, we have
\begin{itemize}
	\item[Case I.] either 
		\ary
		\hat{p}_{M_1-1}\cdots \hat{p}_0[0,y'_0) &\subset (y_1-r_0,y_1 + r_0), \label{8.6.10}\\
		\hat{p}^{-1}_{M_2}\cdots \hat{p}_{-1}^{-1}[y'_0,1) &\subset (y'_1-r_0, y'_1+ r_0);
		\eary
	\item[Case II.] or 
		\ary
		\hat{p}_{M_1-1}\cdots \hat{p}_0[y'_0,1) &\subset (y_1-r_0, y_1+r_0), \\
		\hat{p}^{-1}_{M_2}\cdots \hat{p}_{-1}^{-1}[0,y'_0) &\subset (y'_1-r_0, y'_1 +r_0).
		\eary
\end{itemize}
We will only detail the proof in Case  (I), as the other case follows from a similar argument.

We now define $p_{\lceil \frac{N}{2}\rceil}, \cdots,p_{N-1}$.  Take an arbitrary integer
\ary 
m_1\in (\frac{9\log D (N - \lceil N/2 \rceil)}{9\log D+L_+(f)-L_-(f)},\frac{(9\log D+\eta) (N - \lceil N/2 \rceil)}{9\log D+L_+(f)-L_-(f)}).
\label{8.6.11}
\eary
Define
\aryst
x_1 = g^{M_1}(x_0), && x_2=g^{m_1}(x_1),\\
x_3 =g^{m_2}(x_2), && m_3=N-\lceil \frac{N}{2} \rceil -M_1-m_1-m_2.
\earyst
By direct computations and by letting $N_4$ be sufficiently large, we see that
\begin{equation}
m_3 \in ( \frac{ (L_+(f)-L_-(f)-2\eta) (N - \lceil N/2 \rceil) }{9\log D+L_+(f)-L_-(f)}, \frac{(L_+(f)-L_-(f) - \eta) (N - \lceil N/2 \rceil)}{9\log D+L_+(f)-L_-(f)}). \label{8.6.12}
\end{equation}
Then we have
\ary \label{eq sumofN2m2}
N_2 + m_2 = N - \lceil N/2 \rceil - m_1 - m_3 \in ( 0, \frac{ 2\eta(N - \lceil N/2 \rceil) }{9\log D+L_+(f)-L_-(f)}).
\eary

We define
\ary
p_{\lceil\frac{N}{2}\rceil +i} &=& \hat{p}_i,  \ \forall 0\leq i\leq M_1-1, \label{8.6.13}\\
p_{\lceil\frac{N}{2}\rceil+M_1+i} &=&\tilde{p}_i,  \  \forall 0\leq i\leq m_1-1, \label{8.6.14}\\
p_{\lceil\frac{N}{2}\rceil +M_1+m_1+m_2+i} &=&H( g^i(x_3) ), \ \forall 0\leq i\leq m_3-1, \label{8.6.15}
\eary

By (\ref{8.6.11}) and by letting $N_4$ be sufficiently large,  we have $m_1 > N_3$.  Then by (g2) and \eqref{8.6.9}, we have  
\ary
\tilde{P}_{m_1-1}\cdots\tilde{P}_0(y_1-r_0,y_1+r_0) & \subset& (F^{m_1})_{x_1}(y_1) + e^{\min(\eta,\epsilon_3)m_1}D(\tilde{P}_{m_1-1}\cdots\tilde{P}_0)(y_1)(-r_0,r_0) \nonumber \\
  & \subset&  (F^{m_1})_{x_1}(y_1) + e^{-\epsilon_3 m_1}(-r_0,r_0). \label{8.6.16}
 \eary
 By $r_0 \leq r_2$ and Lemma \ref{lemma 69} (Forward expansion) for $(x_3, m_3)$ in place of $(x, N)$, there exists $y_3\in\mathbb{R}$, such that
\begin{equation}
D(F^{m_3})_{x_3}(y') > e^{-\eta m_3}, \ \forall y'\in (y_3 -D^{-m_3}r_0, y_3+ D^{-m_3}r_0).
\label{8.6.17}
\end{equation}
Without loss of generality, we can choose $y_3$ so that $\vert (F^{m_1+m_2})_{x_1} (y_1) -  y_3\vert < 2$.
By \eqref{eq m2}, we have $(x_2,  (F^{m_1})_{x_1}(y_1) , y_3) \in \Omega_{m_2}(F,  \kappa_1(\Phi(H), \epsilon_0) )$.
 Then by Lemma \ref{lem cancellation} and $m_2 > N_1(\Phi(H), \epsilon_0)$,  we can define $(p_{\lceil\frac{N}{2}\rceil+M_1+m_1},\cdots,  p_{\lceil\frac{N}{2}\rceil+M_1+m_1+m_2-1})$ so that 
\begin{enumerate}
	\item $d_{{\cal H}}(p_{\lceil\frac{N}{2}\rceil+M_1+m_1+i}, H(g^i(x_2)))<\epsilon$ for all $0\leq i\leq m_2-1$,
	\item we have	
	\ary
				P_{\lceil\frac{N}{2}\rceil+M_1+m_1+m_2-1} \cdots P_{\lceil\frac{N}{2}\rceil+M_1}(y_1) = y_3.
			\label{8.6.18}
		\eary
	\end{enumerate}
	Here we denote by $P_{i}$ the unique lift of $\iota_{\cH}(p_i)$ close to $F_{g^{i}(x)}$ for every $\lceil\frac{N}{2}\rceil \leq i  \leq  N  -1$.
Notice that by (1) above, we have
 	\ary
			D^{-m_2} < \|D \iota_{\cH}((p_{\lceil\frac{N}{2}\rceil+M_1+m_1+i})_{i=0}^{m_2-1})\| <D^{m_2}.
			\label{8.6.19}
	\eary

We now estimate $D(P_{N-1}\cdots P_{\lceil\frac{N}{2}\rceil})$ over the interval $[0,y'_0)$. Fix an arbitrary $y'\in  [0,y'_0)$. By (\ref{8.6.10}),(\ref{8.6.13}) we have 
\begin{equation}
	P_{\lceil\frac{N}{2} \rceil +M_1-1}\cdots P_{\lceil\frac{N}{2}\rceil}(y')\in (y_1-r_0, y_1+r_0).\label{8.6.20}
\end{equation}
Then by (\ref{8.6.9}), (\ref{8.6.19}) and (\ref{8.6.18}), we have
\ary
&&P_{\lceil\frac{N}{2}\rceil+M_1+m_1+m_2-1}\cdots P_{\lceil\frac{N}{2}\rceil}(y')
 \nonumber\\
&\in & y_3 + D ^{m_2} e^{\eta m_1} D(P_{\lceil{\frac{N}{2}\rceil +M_1+m_1-1}} \cdots P_{\lceil \frac{N}{2}\rceil+M_1})(y_1) (-r_0, r_0).\label{8.6.21}
\eary
\smallskip
\textbf{(Upper bound)}. It is direct to see that
\begin{equation}\label{8.6.22}
D(P_{\lceil\frac{N}{2}\rceil+M_1-1}\cdots P_{\lceil\frac{N}{2}\rceil})(y') < D^{M_1}. 
\end{equation}
By (\ref{8.6.20}), (g2), (g3), and (\ref{8.6.14}), we have
\ary
&&D(P_{\lceil\frac{N}{2}\rceil+M_1+m_1-1}\cdots P_{\lceil\frac{N}{2}\rceil+M_1})(P_{\lceil\frac{N}{2}\rceil+M_1-1}\cdots P_{\lceil\frac{N}{2}\rceil}(y'))
\label{8.6.23}\\
&<& e^{m_1\epsilon_3} D(P_{\lceil{\frac{N}{2}}\rceil+M_1+m_1-1}\cdots P_{\lceil{\frac{N_2}{2}}\rceil+M_1})(y_1)< e^{-m_1\epsilon_3}. \nonumber
\eary
By (\ref{8.6.12}) and by letting $N_4$ be sufficiently large, we have $m_3 > n_1$.  Then by (\ref{8.6.7}), (\ref{8.6.22}), (\ref{8.6.23}), (\ref{8.6.19}), \eqref{8.6.11} and \eqref{8.6.12}, we have
\aryst
D(P_{N-1}\cdots P_{\lceil\frac{N}{2}\rceil})(y') & <& 
e^{(L_+(f)+\eta)m_3} e^{-m_1\epsilon_3} D^{M_1+m_2},\\
&<&   e^{\frac{1}{3}(L_+(f)-L_-(f))(N-\lceil\frac{N}{2}\rceil)}. 
\earyst
\textbf{(Lower bound)}.  We have
\begin{equation}
D(P_{\lceil\frac{N}{2}\rceil+M_1-1}\cdots P_{\lceil\frac{N}{2}\rceil})(y') >  D^{-M_1}.
\label{8.6.24}
\end{equation}
By (\ref{8.6.23}), it is useful to divide the estimate into the following two cases.
\begin{itemize}
	\item If $D(P_{\lceil{\frac{N}{2}\rceil +M_1+m_1-1}} \cdots P_{\lceil \frac{N}{2}\rceil+M_1})(y_1) > e^{-\frac{L_+(f)-L_-(f)}{4}m_1}$, then by \eqref{8.6.9}, \eqref{8.6.20} and 
	(\ref{8.6.14}), we have
	\aryst
	&& D(P_{\lceil \frac{N}{2} \rceil+M_1+m_1-1}\cdots P_{\lceil\frac{N}{2}\rceil +M_1})(P_{\lceil \frac{N}{2} \rceil+M_1-1}\cdots P_{\lceil\frac{N}{2} \rceil}(y') ) \\
	&>&  
	e^{-m_1\eta}D(P_{\lceil \frac{N}{2} \rceil+M_1+m_1-1}\cdots P_{\lceil\frac{N}{2} \rceil+M_1})(y_1)>  e^{(-\frac{L_+(f)-L_(f)}{4} -\eta)m_1}. 
	\earyst
Then by \eqref{8.6.7} and $m_3 > n_1$, we have
\ary \label{eq lowerbound1}
	 D(P_{N_1}\cdots P_{\lceil\frac{N}{2}}\rceil)(y') 
	> 
	e^{(L_-(f)-\eta)m_3}e^{(-\frac{L_+(f)-L_-(f)}{4} - \eta)m_1} D^{-M_1-m_2}. 
\eary
	\item If $(P_{\lceil \frac{N}{2} \rceil  +M_1+m_1-1}\cdots P_{\lceil\frac{N}{2}\rceil +M_1})(y_1)\leq e^{-\frac{L_+(f)-L_-(f)}{4}m_1}$, then by \eqref{8.6.18}, \eqref{8.6.11} and \eqref{8.6.12} we have
	\aryst
	\mbox{(RHS) of }(\ref{8.6.21}) \subset y_3+ D^{m_2}e^{(\eta-\frac{L_+(f)-L_-(f)}{4})m_1}(-r_0,r_0) \subset y_3+ D^{-m_3}(-r_0,r_0).
	\earyst
\end{itemize}
The last inclusion follows from (\ref{8.6.11}) and (\ref{8.6.12}). 
Then by (\ref{8.6.17}), we have
\aryst
D(P_{N-1} \cdots P_{\lceil\frac{N}{2}\rceil+M_1+m_1+m_2})(P_{\lceil\frac{N}{2}\rceil+M_1+m_1+m_2-1}\cdots P_{\lceil\frac{N}{2}\rceil}(y') ) > e^{-m_3 \eta}.
\earyst
Moreover, by
$ d_{\cal H}(\tilde{p}_i,H(g^{M_1+i}(x_0)))<\epsilon<\epsilon'$ for all $0\leq i\leq m_1-1$, by $m_1>n_1$ and by the choice of $\epsilon'$, we have
\aryst
D(P_{\lceil \frac{N}{2} \rceil  +M_1+m_1-1}\cdots P_{\lceil\frac{N}{2}\rceil +M_1})(y'') > e^{(L_-(f)-2\eta)m_1}, \quad \forall y'' \in \mathbb{R}.
\earyst
By combining the above inequalities with (\ref{8.6.15}), (\ref{8.6.19}) and (\ref{8.6.24}), we obtain
\ary  \label{eq lowerbound2}
D(P_{N-1} \cdots P_{\lceil\frac{N}{2} \rceil})(y') > e^{ - \eta m_3} e^{(L_-(f) - 2\eta) m_1} D^{-  M_1 -m_2}.
\eary

By  (\ref{8.6.11}), (\ref{8.6.12}) and \eqref{eq sumofN2m2},  we can deduce from both \eqref{eq lowerbound1} and \eqref{eq lowerbound2} that
\aryst
D(P_{N-1}\cdots P_{\lceil\frac{N}{2}\rceil})(y') > e^{- (1- \eta )\max(L_+(f),-L_-(f))(N-\lceil \frac{N}{2} \rceil)}. 
\earyst

Now, continue to assume that we are under Case I, then we can define $p_0, \cdots, p_{\lceil\frac{N}{2}\rceil-1}$ in a similar way so that for any $y' \in [y'_0, 1)$, we have
\aryst
|\log | D(\iota_{\cH}((p_i)_{i=0}^{\lceil\frac{N}{2}\rceil-1} ) )^{-1}(y') | | \leq (1- \eta )\max(L_+(f),-L_-(f)) \lceil \frac{N}{2} \rceil.
\earyst
It is then straightforward to verify items (1) and (2) of Lemma \ref{lem mainpropertymovableset} for Case I. The proof for Case II follows from a similar argument.
\end{proof}

\subsection*{Acknowledgement}
{\small
This paper was initiated during the Ph.D. period of Z.Z. The authors thank Artur Avila for useful conversations, and thank Rapha\"el Krikorian for asking the question about qpf Arnold circle maps.
J. W. was partially supported by NSFC (Nos. 12071231, 11971246) and the Fundamental Research Funds for the Central University.
 
}

\end{document}